\documentclass{amsart}
\usepackage{amssymb, latexsym}
\usepackage{enumerate}
\usepackage{footnote}

\theoremstyle{plain}
\newtheorem{thm}{Theorem}

\newtheorem{rem}[thm]{Remark}
\newtheorem{prop}[thm]{Proposition}
\newtheorem{lem}[thm]{Lemma}

\numberwithin{equation}{section} \numberwithin{thm}{section}

\begin{document}

\title[Control Of Sobolev Norms]{ On Control Of Sobolev
Norms For Some Semilinear Wave Equations With Localized Data}

\author{Tristan Roy}
\address{Nagoya University}
\email{tristanroy@math.nagoya-u.ac.jp}

\vspace{-0.3in}

\begin{abstract}

Consider the semilinear wave equations in dimension $3$ with a defocusing and superconformal power-type nonlinearity
and with data lying in the $H^{s} \times H^{s-1}$ ($s < 1$) closure of smooth functions that are compactly supported inside a ball with
fixed radius. We establish new bounds of the Sobolev norms  of the solution. In particular, we prove that the $H^{s}$ norm of the high frequency component of the solution grows like $T^{\sim (1-s)^{2}+}$ in a neighborhood of $s=1$. In order to do that, we perform an analysis in a neighborhood of the cone, using the finite speed of propagation, an almost Shatah-Struwe estimate \cite{shatstruwe}, an almost conservation law and a low-high frequency decomposition \cite{bourgbook,almckstt}.
\thanks{ The notation  $ \sim f(s) $ means that there exists $\alpha(s)$ defined in a neighborhood of $1$ such that  $\sim f(s) := \alpha(s) f(s)$
and $\lim_{s \rightarrow 1-} \alpha(s)  = C$ with $C >0$}
\end{abstract}

\maketitle

\section{Introduction}

In this paper we consider the semilinear wave equations on $\mathbb{R}^{3}$ with a defocusing power-type nonlinearity:

\begin{equation}
\begin{array}{ll}
\partial_{tt} u  - \Delta u  & =  - |u|^{p-1} u \\
\end{array}
\label{Eqn:NlWdat}
\end{equation}
with data $u(0):=u_{0}$, $\partial_{t}u(0):=u_{1}$.\\
The existence of smooth solutions of (\ref{Eqn:NlWdat}) for all time  has received a great deal of
attention from the community. This problem was addressed in \cite{jo} for subcritical powers (i.e $p<5$). The critical power (i.e $p=5$) was solved in \cite{rauch} for small data, in \cite{struwe} for large and radial data and in \cite{grill} for large and general data. No result is known
for the supercritical powers, i.e $p>5$. \\
The next step is to construct solutions of (\ref{Eqn:NlWdat}) with rougher data. In this paper we restrict ourselves to the subcritical and superconformal
exponents, i.e $ 3 \leq p <5$. It is known (see \cite{linsog}) that  (\ref{Eqn:NlWdat}) is locally well-posed in $H^{s} \times H^{s-1}$ for $s > \frac{3}{2} - \frac{2}{p-1}$. By that we mean that

\begin{itemize}

 \item given $(u_{0},u_{1}) \in H^{s} \times H^{s-1}$  there exists $T_{l}>0$ and a unique  $(u,\partial_{t} u)$ lying in a subspace of
 $\mathcal{C}([0,T_{l}],H^{s}) \times \mathcal{C}([0,T_{l}],H^{s-1})$ such that $u$ satisfies the Duhamel formula for all $t \in [0,T_{l}]$, i.e

\begin{equation}
\begin{array}{ll}
u(t) & = \cos{(t D)} u_{0} + \frac{\sin(t D)}{D} u_{1} - \int_{0}^{t} \frac{\sin \left(
(t-t^{'})  D \right)}{D} \left[  |u|^{p-1}(t^{'}) u(t') \right] \, dt^{'} \\
& := \Psi_{t}(u_{0},u_{1})
\end{array}
\label{Eqn:StrongSol}
\end{equation}

\item $(u_{0},u_{1})  \rightarrow \Psi_{t}(u_{0},u_{1}) $ is uniformly continuous in the $H^{s} \times H^{s-1}$ topology

\end{itemize}
Moreover the time of local existence $T_{l}$ depends on the size of the initial data, i.e $T_{l}:= \left( \| (u_{0},u_{1}) \|_{H^{s} \times H^{s-1}} \right)$. Here $H^{s}$ is the standard inhomogeneous Sobolev space i.e $H^{s}$ is the completion of the Schwartz space $\mathcal{S}(\mathbb{R}^{3})$ with respect to the norm

\begin{equation}
\begin{array}{ll}
\| f \|_{H^{s}} &  :=  \| \langle D \rangle^{s} f \|_{L^{2}(\mathbb{R}^{3})},
\end{array}
\nonumber
\end{equation}
where $ D $ is the operator defined by

\begin{equation}
\begin{array}{ll}
\widehat{\langle D \rangle f}(\xi) &  := ( 1 + |\xi|^{2})^{\frac{1}{2}} \hat{f}(\xi)
\end{array}
\nonumber
\end{equation}
and $\hat{f}$ denotes the Fourier transform, i.e

\begin{equation}
\begin{array}{ll}
\hat{f}(\xi) &  := \int_{\mathbb{R}^{3}} f(x) e^{-i x \cdot \xi} \, dx \cdot
\end{array}
\nonumber
\end{equation}
By the local well-posedness theory, the global behavior of $H^{s}$ solutions of (\ref{Eqn:NlWdat}) is closely related to the growth
of the Sobolev norms $\| (u(T),\partial_{t} u(T)) \|_{H^{s} \times H^{s-1}}$ for $T < T_{*}$ where $T_{*}$ is the maximal time of existence.
In particular, if one can find a finite bound of $\| (u(T),\partial_{t} u(T)) \|_{H^{s} \times H^{s-1}}$  for all time $T$, then one can
prove that the $H^{s}$ solutions of (\ref{Eqn:NlWdat}) exist for all time $T$. The equation (\ref{Eqn:NlWdat}) enjoys the following
energy conservation law

\begin{equation}
\begin{array}{ll}
E(u(t)) & := \frac{1}{2} \int_{\mathbb{R}^{3}} |\partial_{t} u(t,x)|^{2} \, dx + \frac{1}{2} \int_{\mathbb{R}^{3}} |\nabla u(t,x)|^{2} \, dx
+ \frac{1}{p+1} \int_{\mathbb{R}^{3}} |u(t,x)|^{p+1} \, dx
\end{array}
\label{Eqn:Nrj}
\end{equation}
Therefore, by using this energy conservation law, we immediately see that $H^{1}$-solutions of
(\ref{Eqn:NlWdat}) exist for all time. It remains to better understand the global behavior of $H^{s}$- solutions
of (\ref{Eqn:NlWdat}) if $s<1$. This question is delicate since there there is no known conservation law at these levels
of regularity. It has been studied in \cite{kenponcevega,gallagplanch,bahchemin,triroyrad,triroygen} (see \cite{miaozhang} for higher
dimensions). To our knowledge, the best results regarding the optimal index of regularity for which the solution exists for $p=3$ and
for all time $T$ are obtained in \cite{triroygen} for general data ($s>\frac{13}{18}$) and in \cite{triroyrad} for radial data ($s>\frac{7}{10}$).

The purpose of this paper is to improve the bounds of the $H^{s}$ norms of the solution for a class of rough
and localized data, that is, $(u_{0},u_{1}) \in Cl \left( \mathcal{C}^{\infty}_{c} (B(0,R)), H^{s} \right) \times
Cl \left( \mathcal{C}^{\infty}_{c}(B(0,R)), H^{s-1} \right)$, where $R>0$ is an arbitrary but fixed positive number and
$Cl \left( \mathcal{C}^{\infty}_{c}(B(0,R)), H^{s} \right) \times Cl \left(  \mathcal{C}^{\infty}_{c}(B(O,R)), H^{s-1} \right)$ is the closure of smooth
and compactly supported functions inside the ball $B(O,R):= \left\{ x \in \mathbb{R}^{3}: |x| < R   \right\}$ with respect
to the $H^{s} \times H^{s-1}$ topology. Our main result is \footnote{Here $P_{<1} f$ (resp. $P_{>1} f$) denotes the low frequency part (resp. the high frequency part) of a function $f$, i.e $\widehat{P_{<1}f}(\xi)  := \phi (\xi) \hat{f}(\xi)$, $\widehat{P_{>1}f}(\xi)
:=  \left( 1- \phi(\xi) \right) \hat{f}(\xi)$ with $\phi(\xi)$ a smooth, real, radial, nonincreasing function that is equal to one on
$B(O,1)$ and that is supported on $B(O,2)$. }:

\begin{thm}
Let $u$ be the solution of (\ref{Eqn:NlWdat}) with data $( u_{0},u_{1} ) \in
Cl \left( \mathcal{C}^{\infty}_{c} (B(0,R)), H^{s} \right) \times
Cl \left( \mathcal{C}^{\infty}_{c}(B(0,R)), H^{s-1} \right) $.  Let

\begin{equation}
\begin{array}{ll}
\theta & := \frac{4p-12}{(p-1)(7-p)}
\end{array}
\nonumber
\end{equation}
Then, if $1 > s > s_{p}:= 1 - \frac{(5-p)(1- \theta)}{2(p+1)} $ and
$T \geq 1$

\begin{equation}
\begin{array}{ll}
\| (P_{>1} u(T), \partial_{t}  u(T)) \|^{2}_{H^{s} \times H^{s-1}} & \lesssim
T^{ \frac{4(1-s)^{2} + }
{
\left(
\frac{2 \theta +  p-1 }{1-\theta} s - \left( \frac{2 \theta + p-1}{1- \theta} - \frac{5-p}{2}
\right) \right) \left(  \frac{5-p}{2} - \frac{(1-s)(p+1)}{1- \theta}   \right) }}
\end{array}
\label{Eqn:MainIneqHigh}
\end{equation}
and

\begin{equation}
\begin{array}{ll}
\| P_{<1} u(T) \|^{2}_{H^{s}} T^{-\frac{3p-5}{p+1}} & \lesssim
T^{\frac{4(1-s)+} {(p+1) \left( \frac{5-p}{2} - \frac{(1-s)(p+1)}{1- \theta}  \right)  }}
\end{array}
\label{Eqn:MainIneqLow}
\end{equation}


\label{Thm:Main}
\end{thm}

\textit{Comparison with the existing results}: we shall compare the results with \cite{kenponcevega,triroygen,triroyrad}.

\begin{itemize}

\item \underline{$p=3$} \footnote{the cubic power has attracted much attention from the community: see recent work regarding
probabilistic well-posedness in \cite{burqtzvet}}: then one gets from (\ref{Eqn:MainIneqHigh}) and (\ref{Eqn:MainIneqLow})

\begin{equation}
\begin{array}{ll}
\left\| ( P_{>1} u(T),  \partial_{t} u(T) ) \right\|^{2}_{H^{s} \times H^{s-1}} & \lesssim
 T^{\frac{4(1-s)^{2}}{(4s-3)(2s-1)}+}
\end{array}
\nonumber
\end{equation}
and

\begin{equation}
\begin{array}{ll}
\left\| P_{<1} u(T) \right\|^{2}_{H^{s}} T^{-1} & \lesssim
T^{\frac{1-s } {4s-3} +}
\end{array}
\nonumber
\end{equation}

\textit{Comparison with \cite{triroygen,triroyrad}}. We recall the method used in these papers in order to estimate
$ \| P_{>1} u(T), \partial_{t} u (T)  \|^{2}_{H^{s} \times H^{s-1}} $  and $ \| P_{<1} u(T) \|^{2}_{H^{s}}$.
If $s=1 $ then it is pretty easy to estimate these norms by using the conservation of the energy (\ref{Eqn:Nrj}). If
$s < 1$ then one cannot use the energy by itself since it can be infinite. Instead one introduces the following
functional

\begin{equation}
\begin{array}{l}
E(I_{N} u(t)) := \frac{1}{2} \int_{\mathbb{R}^{3}} |\partial_{t} I_{N} u(t,x)|^{2} \, dx +
\frac{1}{2} \int_{\mathbb{R}^{3}} |\nabla I_{N} u (t,x) |^{2} \, dx +  \frac{1}{p+1} \int_{\mathbb{R}^{3}} |I_{N}u(t,x)|^{p+1} \, dx,
\end{array}
\label{Eqn:DefEIu}
\end{equation}
where the multiplier $I_{N}$ is defined by

\begin{equation}
\begin{array}{ll}
\widehat{I_{N} f}(\xi) & = m_{N}(\xi) \hat{f}(\xi),
\end{array}
\nonumber
\end{equation}
with $m_{N}(\xi) := \eta \left( \frac{\xi}{N}  \right)$, $\eta$ is a smooth, radial,
nonincreasing function in $|\xi|$ such that

\begin{equation}
\begin{array}{ll}
\eta(\xi) & :=
\left\{
\begin{array}{l}
1, \, |\xi| \leq 1 \\
\frac{1}{|\xi|^{1-s}}, \, |\xi| \geq 2,
\end{array}
\right.
\end{array}
\nonumber
\end{equation}
and $N \gg 1$ is a dyadic number playing the role of a parameter to be chosen. This is the $I$-method, designed in
\cite{almckstt} and inspired from the Fourier truncation method, designed in \cite{bourgbook}. The main interest of
introducing this multiplier $I_{N}$ is that (\ref{Eqn:DefEIu}) is finite in $H^{s} \times H^{s-1}$ for $s < 1$. Moreover
the variation of (\ref{Eqn:DefEIu}) is expected to be slow for $N \gg 1$ , since the multiplier $I_{N}$ is close to the
identity. Once we have estimated the variation of $E(I_{N} u(t))$ for a well-chosen $N \gg 1$, we can easily estimate for
data $(u_{0},u_{1}) \in H^{s} \times H^{s-1}$  the norms of the solution through the following inequalities

\begin{equation}
\begin{array}{ll}
\| ( P_{>1} u(T), \partial_{t} u(T) ) \|^{2}_{H^{s} \times H^{s-1}} & \lesssim  \sup_{t \in [0,T]} E(I_{N}u(t))
\end{array}
\label{Eqn:HighHsNrj}
\end{equation}

\begin{equation}
\begin{array}{ll}
\| P_{<1} u(T) \|^{2}_{H^{s}} & \lesssim  T^{2} \sup_{t \in [0,T]}  \| \partial_{t} I_{N} u(t) \|^{2}_{L^{2}} \\
& \lesssim  T^{2} \sup_{ t \in [0,T]} E (I_{N} u(t))
\end{array}
\label{Eqn:EstLocFreqProt1}
\end{equation}
In the case of data $(u_{0},u_{1}) \in Cl \left(  \mathcal{C}_{c}^{\infty} (B(O,R)), H^{s} \right) \times
Cl \left( \mathcal{C}_{c}^{\infty} (B(O,R)), H^{s-1} \right)$, we can upgrade (\ref{Eqn:EstLocFreqProt1}) using
the finite speed of propagation and  (\ref{Eqn:SpatSmooth3}) in the following fashion

\begin{equation}
\begin{array}{ll}
\| P_{ < 1} u(T) \|^{2}_{H^{s}} T^{-\frac{3}{2}} & \lesssim T^{-\frac{3}{2}} \| I_{N} u(T) \|^{2}_{L^{2}} \\
& \sim T^{-\frac{3}{2}} \| I_{N} u(T) \|^{2}_{L^{2} (B(0,R+1+T))} \\
& \lesssim \| I_{N} u(T) \|^{2}_{L^{4} (B(0,R+1+T))} \\
& \lesssim \sup_{t \in [0,T] } E^{\frac{1}{2}} (I_{N} u(t))
\end{array}
\label{Eqn:EstLowFreqUpg}
\end{equation}
By using this method and by an adapted linear-nonlinear decomposition, it was proved in \cite{triroygen}  that
the solution of (\ref{Eqn:NlWdat}) exists globally for $s>\frac{13}{18} $ and, moreover,

\begin{equation}
\begin{array}{ll}
\sup_{t \in [0,T]} E(I_{N}u(t)) & \lesssim  T^{\frac{8(1-s)}{18s-13}+}
\end{array}
\nonumber
\end{equation}
Hence, for data  $(u_{0},u_{1}) \in Cl \left(  \mathcal{C}_{c}^{\infty} (B(O,R)), H^{s} \right) \times
Cl \left( \mathcal{C}_{c}^{\infty} (B(O,R)), H^{s-1} \right)$, we have

\begin{equation}
\begin{array}{l}
\|  ( P_{>1} u(T), \partial_{t} u(T) ) \|^{2}_{H^{s} \times H^{s-1}} \lesssim
T^{\frac{8(1-s)}{18s-13}+}
\end{array}
\nonumber
\end{equation}
and

\begin{equation}
\begin{array}{ll}
\|  P_{ < 1} u(T) \|^{2}_{H^{s}} T^{ - \frac{3}{2}} & \lesssim  T^{\frac{4(1-s)}{18s-13}+}
\end{array}
\nonumber
\end{equation}
So our result is an improvement up to $s > \frac{51 + \sqrt{17}}{68} \approx 0.81$.

\textit{Comparison with \cite{triroyrad}}. Under the additional assumption of radial symmetry, it was proved in \cite{triroyrad} by the use of a weighted Morawetz estimate and a radial Sobolev inequality that the solution of (\ref{Eqn:NlWdat}) exists globally for $s>\frac{7}{10}$ and, moreover

\begin{equation}
\begin{array}{l}
\sup_{t \in [0,T]} E(I_{N}u(t))  \lesssim  T^{\frac{2(1-s)}{2s-1}+} , \, 1 > s \geq \frac{5}{6} \\
\sup_{t \in [0,T]} E(I_{N}u(t)) \lesssim   T^{\frac{4(1-s)}{10s -7} +}, \,  \frac{5}{6} \geq  s > \frac{7}{10}
\end{array}
\nonumber
\end{equation}
So we get in a similar fashion

\begin{equation}
\begin{array}{l}
\| ( P_{>1} u(T), \partial_{t} u(T) ) \|^{2}_{H^{s} \times H^{s-1}}  \lesssim    T^{\frac{2(1-s)}{2s-1}+} , \, s \geq \frac{5}{6} \\
\| ( P_{>1} u(T), \partial_{t} u(T) ) \|^{2}_{H^{s} \times H^{s-1}}  \lesssim  T^{\frac{4(1-s)}{10s -7} +}, \, \frac{5}{6} \geq  s > \frac{7}{10}
\end{array}
\nonumber
\end{equation}
and

\begin{equation}
\begin{array}{ll}
\| P_{< 1} u(T) \|^{2}_{H^{s}} T^{-\frac{3}{2}} \lesssim T^{\frac{1-s}{2s-1}+},  \,   1 > s \geq \frac{5}{6} \\
\| P_{<1} u(T) \|^{2}_{H^{s}} T^{-\frac{3}{2}} \lesssim  T^{\frac{2(1-s)}{10s -7} +},  \,  \frac{5}{6} \geq s > \frac{7}{10}
\end{array}
\nonumber
\end{equation}
So the improvement holds up to $ s>\frac{5}{6} $.

\textit{Comparison with \cite{kenponcevega}}. The comparison can only be partial since the authors considered data in slightly different spaces, i.e
$(u_{0} , u_{1}) \times \dot{H}^{s} \cap L^{4} \times \dot{H}^{s-1}$. It was proved in \cite{kenponcevega} that the solution exists globally for $s>\frac{3}{4}$ and, moreover,

\begin{equation}
\begin{array}{ll}
u(t) & = \cos{(tD)} u_{0} + \frac{\sin{(tD)}}{D} u_{1} + z(t)
\end{array}
\nonumber
\end{equation}
with

\begin{equation}
\begin{array}{ll}
\sup_{t \in [0,T]} \left( \| \partial_{t} z(t) \|_{L^{2}}, \| \nabla z(t) \|_{L^{2}}, \| z(t) \|^{2}_{L^{4}} \right) & \lesssim T^{\frac{1-s}{4s-3}+}
\end{array}
\label{Eqn:ResKPVzp3}
\end{equation}
Consider data $ (u_{0},u_{1}) \in Cl \left( \mathcal{C}_{c}^{\infty} (B(O,R)), H^{s}), \right)  \times
Cl \left( \mathcal{C}_{c}^{\infty} (B(O,R)), H^{s-1} \right) $, that is $(u_{0},u_{1})$ lies in the closure of smooth and
compactly support functions inside the ball $B(O,R)$ with respect to the $H^{s} \times \dot{H}^{s-1}$ topology. Notice that

\begin{equation}
\begin{array}{l}
Cl \left( \mathcal{C}_{c}^{\infty} (B(O,R)), H^{s} \right)  \times
Cl \left( \mathcal{C}_{c}^{\infty} (B(O,R)), \dot{H}^{s-1} \right) \subset  ( \dot{H}^{s} \cap L^{4} \times \dot{H}^{s-1} ) \cap ( H^{s}
\times H^{s-1} )
\end{array}
\nonumber
\end{equation}
Using (\ref{Eqn:ResKPVzp3}) and  ( \ref{Eqn:EstLowFreqUpg} ) we have

\begin{equation}
\begin{array}{ll}
\| P_{<1} u(T) \|^{2}_{H^{s}}  & \lesssim \left \| P_{<1} \cos{(TD)} u_{0} \right \|^{2}_{H^{s}} + \left \| P_{< 1} \frac{\sin{(TD)}}{D} u_{1} \right\|^{2}_{H^{s}}
+ \left\| P_{<1} z(T)  \right\|^{2}_{H^{s}} \\
& \lesssim \max{ (  T^{2},  T^{\frac{3}{2}} T^{\frac{1-s}{4s-3}+}  ) }
\end{array}
\label{Eqn:ControlKPVLowFreq}
\end{equation}
Moreover

\begin{equation}
\begin{array}{ll}
\| (P_{>1} u(T),  \partial_{t} u(T)) \|^{2}_{H^{s} \times H^{s-1}} & \lesssim  T^{\frac{2(1-s)}{4s-3}+}
\end{array}
\nonumber
\end{equation}
So the improvement holds up to $s > s_{3} := \frac{3}{4}$.
\vspace{5 mm}

\item \underline{$3<p<5$}  It was proved in \cite{kenponcevega} that the solution exists globally for
$1 > s  > s_{p} := \frac{26 p - 3 p^{2} - 39}{2(p-1)(7-p)}$ with  data  $(u_{0}, u_{1}) \in  \dot{H}^{s} \cap L^{p+1}
\times  \dot{H}^{s-1}$, and, moreover

\begin{equation}
\begin{array}{ll}
\sup_{t \in [0,T]}  \left(  \| \partial_{t} z(t) \|_{L^{2}},  \| \nabla z(t) \|_{L^{2}}, \| z(t) \|^{\frac{p+1}{2}}_{L^{p+1}} \right)
& \lesssim T^{\frac{1-s}{ 1- s - \beta }+}
\end{array}
\label{Eqn:ResKPVzp}
\end{equation}
with $\beta:= \frac{p-3}{2} + p(1-s) - 1 + \frac{2(1-s)(p-1)}{5-p}$. Consider data
$ (u_{0},u_{1}) \in Cl \left( \mathcal{C}_{c}^{\infty} (B(O,R)), H^{s}, \right)  \times
Cl \left( \mathcal{C}_{c}^{\infty} (B(O,R)), \dot{H}^{s-1} \right) $. Again, notice that

\begin{equation}
\begin{array}{l}
Cl \left( \mathcal{C}_{c}^{\infty} (B(O,R)), H^{s} \right)  \times
Cl \left( \mathcal{C}_{c}^{\infty} (B(O,R)), \dot{H}^{s-1} \right) \subset  ( \dot{H}^{s} \cap L^{p+1} \times \dot{H}^{s-1} ) \cap ( H^{s}
\times H^{s-1} )
\end{array}
\nonumber
\end{equation}
By the same token as (\ref{Eqn:ControlKPVLowFreq}) and by modifying slightly (\ref{Eqn:EstLowFreqUpg})

\begin{equation}
\begin{array}{ll}
\| P_{< 1} u(T) \|^{2}_{H^{s}} & \lesssim \max  \left(  T^{2}, T^{\frac{3(p-1)}{p+1}} T^{\frac{4(1-s)}{(p+1)(1-s-\beta)}+} \right) \cdot
\end{array}
\nonumber
\end{equation}
Moreover

\begin{equation}
\begin{array}{ll}
\left\| \left( P_{>1} u(T), \partial_{t} u(T) \right) \right\|^{2}_{H^{s} \times H^{s-1}} & \lesssim T^{\frac{2(1-s)}{1-s -\beta} +}
\end{array}
\nonumber
\end{equation}
So the improvement holds up to $s> s_{p}$.

\end{itemize}

\vspace{5 mm}

\textbf{Conclusion}. \\

\begin{itemize}

\item  $ \| ( P_{>1} u(T), \partial_{t} u(T) ) \|^{2}_{H^{s} \times H^{s-1}}$ grows more slowly  in a neighborhood of $s=1$ for data
$ (u_{0},u_{1}) \in Cl \left( \mathcal{C}_{c}^{\infty} (B(O,R)), H^{s} \right)  \times Cl \left( \mathcal{C}_{c}^{\infty} (B(O,R)), \dot{H}^{s-1} \right) $, $ Cl \left( \mathcal{C}_{c}^{\infty} (B(O,R)), H^{s} \right)  \times Cl \left( \mathcal{C}_{c}^{\infty} (B(O,R)), H^{s-1} \right)$ than
those found in \cite{triroygen,triroyrad}, \cite{kenponcevega} respectively. Indeed, it grows like $T^{ \sim (1-s)^{2} +}$ instead of
$T^{\sim (1-s)+}$ \footnote{The notation $x$ grows like $T^{y}$ in a neighborhood of $s=1$ means that there exists $\alpha(s)$
defined in a neighborhood of $s=1$ such that $x = O \left( T^{y + \alpha(s)} \right)$ and $\lim_{s \rightarrow 1^{-}} \alpha(s) =0$}.
If $5>p>3$ then the improvement holds up to $s_{p}$.

\item $ \| P_{<1} u(T) \|^{2}_{H^{s}}$ grows more slowly in a neighborhood of $s=1$ for data
$ (u_{0},u_{1}) \in Cl \left( \mathcal{C}_{c}^{\infty} (B(O,R)), H^{s} \right)  \times Cl \left( \mathcal{C}_{c}^{\infty} (B(O,R)), \dot{H}^{s-1} \right) $, \\ $ Cl \left( \mathcal{C}_{c}^{\infty} (B(O,R)), H^{s} \right)  \times Cl \left( \mathcal{C}_{c}^{\infty} (B(O,R)), H^{s-1} \right)$ than
those found in \cite{triroygen,triroyrad}, \cite{kenponcevega} respectively. Indeed it grows like $ T^{\frac{3p-5}{p+1} +}$ in a neighborhood
of $s=1$ instead of $T^{2}$. If $5>p>3$ then the improvement holds up to $s_{p}$.

\end{itemize}

\begin{rem}
\label{Rem:comp}
If $3 <p<5$ notice that, to our knowledge, the use of the Morawetz estimate and the
use of the radial Sobolev inequality have not been implemented for radial data; the adapted linear-nonlinear decomposition  has
not been implemented for general data but we expect the same phenomena to occur as $p=3$, i.e the improvement should hold for
$\tilde{s}_{p} < s < 1$ with $\tilde{s}_{p}$ a number satisfying $ 1 > \tilde{s}_{p} > s_{p}$.
\end{rem}

We set some notation that appear throughout the paper. \\
Let $W$ be the set of wave-admissible points, i.e

\begin{equation}
\begin{array}{l}
W  := \left\{ (x,y) \in \mathbb{R}^{2}, \, (x,y) \in (2, \infty] \times [2, \infty), \, \frac{1}{x} + \frac{1}{y} \leq \frac{1}{2} \right\}
\end{array}
\nonumber
\end{equation}
Let $\tilde{W}$ be the dual set of $W$, i.e

\begin{equation}
\begin{array}{l}
\tilde{W} := \left\{ (x',y') \in \mathbb{R}^{2}, \, \exists (x,y) \in W \, s.t \, \frac{1}{x} + \frac{1}{x'}=1, \,
\frac{1}{y} + \frac{1}{y'} =1 \right\}
\end{array}
\end{equation}
Given $m \in [0,1]$, we say that $(q,r)$ is $m-$ wave admissible if

\begin{itemize}
\item $(q,r) \in W$
\item $(q,r)$ satisfies $\frac{1}{q} + \frac{3}{r}= \frac{3}{2} -m$
\item  $q>2+$ if $m=1$
\end{itemize}
Given $(x_{0},R_{0}) \in \mathbb{R}^{3} \times \mathbb{R}^{+}$, let
$B(x_{0},R_{0}):= \{ x \in \mathbb{R}^{3}: |x-x_{0}| < R_{0} \}$. Let $\chi_{R_0}$ be a smooth function
supported on $B \left(0,R_0 + \frac{1}{2} \right)$ and such that  $\chi_{R_0}(x)=1$ if $|x| \leq R_0$. We say that $x \lesssim y$ if there exists
$0 < C:=C (  \| (u_{0},u_{1}) \|_{H^{s} \times H^{s-1}},R ) $  such that $x \leq C y$. We say that
$\tilde{C}$ is the constant determined by $x \lesssim y$ if $\tilde{C}$ is the smallest constant such
that $x \leq C y $ holds. More generally, given
$n \geq 1$ and $(a_{1},....,a_{n}) \in \mathbb{R}^{n}$, we
say that $x \lesssim_{a_{1},...,a_{n}} y$ if there exists \\
$C:=C(a_{1},...,a_{n},\| (u_{0},u_{1}) \|_{H^{s} \times H^{s-1}},R) >0$ such that
$x \leq C y$. We say that $x \lesssim_{\infty-} y^{\infty -}$ if for every $q \geq 1$,
there exists $C:=C(q,\| (u_{0},u_{1}) \|_{H^{s} \times H^{s-1}},R) >0$ such that $x \leq C y^{q}$. We say that $x \ll y$ if there exists
$0 < c:=c (\| (u_{0},u_{1}) \|_{H^{s} \times H^{s-1}},R ) \ll 1$ such that $x \leq c y$. In a similar fashion we extend this definition to
$x \ll_{a_{1},...,a_{n}} y$. \\
Some estimates that we establish throughout the paper require a Paley-Littlewood decomposition. We set it up now.
Let $\phi(\xi)$ be a smooth, real, radial, nonincreasing function that is equal to one $1$ on $B(O,1)$ and that is supported
on $B(O,2)$. Let $\psi$ denote the function $\psi(\xi):= \phi(\xi) - \phi(2 \xi)$.  Let $\tilde{\psi}$ denote the
function $\tilde{\psi}(\xi):= \phi (\frac{\xi}{8}) - \phi (8 \xi)$. If $(M,M_{1},M_{2}) \in (2^{\mathbb{N}^{*}})^{3}$
are three dyadic numbers such that $M_{2} \geq M_{1}$ then

\begin{equation}
\begin{array}{ll}
\widehat{P_{\leq M} f}(\xi) &  := \phi \left( \frac{\xi}{M} \right) \hat{f}(\xi) \\
\widehat{P_{M}f}(\xi) & := \psi \left( \frac{\xi}{M} \right) \hat{f}(\xi)  \\
\widehat{P_{\ll M} f} (\xi) & := \widehat{P_{ \leq  \frac{M}{128}} f} (\xi) \\
\widehat{P_{\gtrsim M} f} (\xi) & := \hat{f}(\xi) - \widehat{P_{\ll M} f}(\xi) \\
\tilde{\phi}(\xi) & := \phi \left( \frac{\xi}{4} \right) - \phi(4\xi) \\
\widehat{\tilde{P}_{M} f} (\xi) & = \tilde{\psi} \left( \frac{\xi}{M} \right) \widehat{f}(\xi) \\
P_{M_{1} \leq \cdot \leq M_{2}} f  & := P_{ \geq M_{2}} f - P_{< M_{1}} f \\
\end{array}
\nonumber
\end{equation}
Notice that $f=P_{\ll M} f + P_{\gtrsim M} f$ and that $\tilde{P}_{M} P_{M} = P_{M}$.
Let

\begin{equation}
\begin{array}{ll}
K_{R_{0}}(J) & := \left\{ (t,x): t \in J, \, t > |x| - R_{0}  \right\} \\
\partial K_{R_{0}}(J) & := \left\{ (t,x): t \in J, \, t=|x| - R_{0}  \right\} \\
K_{R_{0}}^{c}(J) & := \left\{ (t,x): t \in J, \, t < |x| - R_{0}   \right\}
\end{array}
\nonumber
\end{equation}
and

\begin{equation}
\begin{array}{l}
E(I_{N}u(t))  :=  \frac{1}{2} \int_{\mathbb{R}^{3}} |\partial_{t} I_{N} u (t,x)|^{2} \, dx +
\frac{1}{2} \int_{\mathbb{R}^{3}} |\nabla I_{N} u(t,x)|^{2} \, dx + \frac{1}{p+1} \int_{\mathbb{R}^{3}} |I_{N} u(t,x)|^{p+1} \, dx \\
E_{R_{0},sec}(I_{N} u(t))  :=
\begin{array}{l}
\frac{1}{2} \int_{|x| \leq t + R_{0} } |\partial_{t} I_{N} u(t,x)|^{2} \, dx +
\frac{1}{2} \int_{|x| \leq t + R_{0}} |\nabla I_{N} u(t,x)|^{2} \, dx \\
 + \frac{1}{p+1} \int_{|x| \leq t + R_{0}} |I_{N} u(t,x)|^{p+1} \\
\end{array} \\
E_{R_{0},ext}(I_{N} u(t))  :=
\begin{array}{l}
\frac{1}{2} \int_{|x| > t + R_{0}} |\partial_{t} I_{N} u(t,x)|^{2} \, dx +
\frac{1}{2} \int_{|x| > t + R_{0}} |\nabla I_{N} u (t,x)|^{2} \, dx \\
 + \frac{1}{p+1} \int_{|x| > t + R_{0}} |I_{N} u(t,x)|^{p+1} \, dx
\end{array}  \\
Flux(I_{N} u, \partial K_{R_{0}}([a,b]))  := \frac{1}{ \sqrt{2}} \int_{ \partial K_{R_{0}}([a,b])}
 \frac{1}{2} \left| \frac{ \nabla I_{N} u \cdot x}{|x|} + \partial_{t} I_{N} u  \right|^{2} + \frac{| I_{N} u|^{p+1}}{p+1} \, d \sigma
\end{array}
\nonumber
\end{equation}
Given $J$ and interval and $f$ a differentiable in time and smooth function let

\begin{equation}
\begin{array}{ll}
Z_{m,s} (J,f) & := \sup_{(q,r)-m- wave \, admissible} \| \partial_{t} D^{-m} I_{N} f \|_{L_{t}^{q} L_{x}^{r}(J)} + \| D^{1-m} I_{N} f \|_{L_{t}^{q} L_{x}^{r}(J)}
\end{array}
\nonumber
\end{equation}
and

\begin{equation}
\begin{array}{ll}
Z(J,f) & := \sup_{m \in [0,1]} Z_{m,s}(J,f)
\end{array}
\nonumber
\end{equation}
If we work with the same parameter $N$, then we forget it in all the expressions where it appears in order to simplify the notation
and we write $I$ for $I_{N}$.\\
Let $R':=R+1$. Let $s_{c} := \frac{3}{2} - \frac{2}{p-1} $. \\ \\

With this notation in mind, we now recall two propositions. The Strichartz estimates
can be stated as follows:

\begin{prop}{(Strichartz estimates) (See \cite{linsog}.)}
Assume that $u$ satisfies the following wave equation on $\mathbb{R}^{3}$

\begin{equation}
\left\{
\begin{array}{ll}
\partial_{tt} u - \triangle u & = G \\
u(0,x) & := u_{0}(x) \\
\partial_{t} u(0,x) & := u_{1}(x)
\end{array}
\right.
\nonumber
\end{equation}
Let $T  \geq 0$. Then, if $m \in [0,1]$

\begin{equation}
\begin{array}{l}
\| u \|_{L_{t}^{q} L_{x}^{r} (J)} + \| \partial_{t} D^{-1} u \|_{L_{t}^{q} L_{x}^{r} (J)}
+ \| u \|_{L_{t}^{\infty} \dot{H}^{m} (J)} + \| \partial_{t} u \|_{ L_{t}^{\infty} \dot{H}^{m} (J)} \\
\lesssim \| u_{0} \|_{\dot{H}^{m}} + \| u_{1} \|_{\dot{H}^{m-1}} + \| F \|_{L_{t}^{\tilde{q}} L_{x}^{\tilde{r}}(J)}
\end{array}
\nonumber
\end{equation}
under the following assumptions:

\begin{itemize}

\item $(q,r)$ is $m$-wave admissible
\item $(\tilde{q}, \tilde{r})$ satisfies the following conditions:

\begin{enumerate}
\item $(\tilde{q},\tilde{r}) \in \tilde{W}$
\item $\frac{1}{\tilde{q}} + \frac{3}{\tilde{r}} - 2 = \frac{1}{q} + \frac{3}{r} $
\end{enumerate}

\end{itemize}
\end{prop}
The second proposition shows that the initial mollified energy at time $0$ is finite in $H^{s} \times H^{s-1}$
and in fact bounded by $N^{2(1-s)}$:

\begin{prop}{(Initial mollified energy at time $0$ is bounded by $N^{2(1-s)}$). (See \cite{triroyrad}.) }
There exists $C_{E}:=C_{E}(\|(u_{0},u_{1})\|_{H^{s} \times H^{s-1}}) >0$ such that

\begin{equation}
\begin{array}{ll}
E(Iu_{0}) & \leq C_{E} N^{2(1-s)} \cdot
\end{array}
\label{Eqn:BoundInitNrj}
\end{equation}
\label{prop:BoundInitNrj}
\end{prop}

Now we explain the main ideas of this paper. It is well-known that the long-time behavior of solutions of semilinear wave equations with a defocusing
nonlinearity is closely related to the Morawetz-type decay estimates. In \cite{triroyrad}, a mollified variant of the Morawetz-Strauss estimate
\cite{morstr} was used, namely

\begin{equation}
\begin{array}{ll}
\int_{0}^{T} \int_{\mathbb{R}^{3}} \frac{|u|^{4}}{|x|} \, dx dt & \lesssim E,
\end{array}
\label{Eqn:WeightMor}
\end{equation}
in the study of the long-time behavior of solutions of (\ref{Eqn:NlWdat}) with $p=3$ and data in $H^{s} \times H^{s-1}$, $s<1$. Under the assumption of the radial symmetry, one was able to control a mollified variant of

\begin{equation}
\begin{array}{ll}
\int_{0}^{T} \int_{\mathbb{R}^{3}} |u|^{6}(t,x) \, dx \, dt & \lesssim E^{2},
\end{array}
\label{Eqn:Lt6Lx6}
\end{equation}
by combining (\ref{Eqn:WeightMor}) with a radial Sobolev estimate, namely

\begin{equation}
\begin{array}{ll}
|u(t,x)| & \lesssim \frac{\| u(t) \|_{\dot{H}^{1}}}{|x|^{\frac{1}{2}}} \cdot
\end{array}
\nonumber
\end{equation}
Notice that a mollified variant of (\ref{Eqn:Lt6Lx6}) is pretty useful for regularity purposes since it is a decay estimate with respect to a mollified variant
of the energy that measures the smoothness of the solution. Indeed, if we ignore the integration with respect to time in (\ref{Eqn:Lt6Lx6}), we have a decay of the form $\| u(t) \|^{6}_{L^{6}}
\lesssim E^{2}$, which is better than the Sobolev embedding $\| u(t) \|^{6}_{L^{6}} \lesssim \| \nabla u(t) \|^{6}_{L^{2}} \lesssim E^{3}$.
Therefore it is useful to estimate the $H^{s}$ norms of the solution ``far'' from $s=1$ \footnote{More precisely, the computations show
that it is useful close to $s=\frac{7}{10}$: see \cite{triroyrad} for more information}. But notice that (\ref{Eqn:WeightMor}) is a
weighted estimate and the decay is slow with respect to time, and even after combining this inequality with the radial Sobolev inequality, one loses
integrability in time. In \cite{triroygen}, an adapted linear-nonlinear decomposition based upon the fact that the nonlinear part of the solution is smoother than $H^{s}$ was performed but this tool can only be used for regularity purposes and it does not yield
information regarding the asymptotic behavior of the solution, in particular close to $s=1$. In the study of the energy-critical wave equation
(i.e $p=5$), a Morawetz-type estimate using the scaling multiplier inside the cone was used, namely \footnote{see for example \cite{bahger}: here we assume
that we work with large energy, i.e $E \gtrsim 1$. }

\begin{equation}
\begin{array}{ll}
\int_{|x| \leq T + R} |u|^{6} (T,x) \, dx & \lesssim \frac{R}{T+R} E + E_{R,sec}(u(T)) - E_{R,sec}(u(0)) \\
& + ( E_{R,sec} (u(T)) - E_{R,sec}(u(0)) )^{\frac{1}{3}}
\end{array}
\nonumber
\end{equation}
This estimate with general data is a weak decay since it is only valid inside the cone and depends on the flux

\begin{equation}
\begin{array}{ll}
Flux( u, \partial K_{R} ([0,T]) )  & := \frac{1}{\sqrt{2}} \int_{\partial K_{R} ([0,T])} \frac{1}{2}
\left| \frac{\nabla u \cdot x}{|x|} + \partial_{t} u   \right|^{2}  + \frac{|u|^{6}}{6} \, d \sigma \\
& = E_{R,sec}(u(T)) - E_{R,sec} (u(0))
\end{array}
\nonumber
\end{equation}
But, if we work with compactly supported data inside the ball $B(O,R)$, then it is much stronger since the flux on the boundary vanishes and, by finite speed of propagation, the solution
is localized inside a cone. Getting back to (\ref{Eqn:NlWdat}), it is worth trying to establish a decay estimate by using the same multiplier for these data. One finds that, for $ 3 \leq p \leq 5$,

\begin{equation}
\begin{array}{ll}
\int_{|x| \leq T+R} |u|^{p+1} (T,x) \, dx & \lesssim  \frac{R}{T+R} E \cdot
\end{array}
\label{Eqn:SubcritMor}
\end{equation}
Notice that this estimate, unlike (\ref{Eqn:WeightMor}), is ``unweighted'' and pointwise in time. Therefore, it can
be used everywhere. The next step is to find the right framework in which we can use this estimate in rougher spaces, i.e
$H^{s} \times H^{s-1}$, $s<1$. It seems natural to choose data
$(u_{0},u_{1}) \in Cl \left(  \mathcal{C}_{c}^{\infty} (B(O,R)), H^{s} \right) \times
Cl \left( \mathcal{C}_{c}^{\infty} (B(O,R)), H^{s-1} \right)$. Then we would like to use a low-high frequency decomposition \cite{bourgbook,almckstt}
in order to estimate the $H^{s}$ norms of the solution. By introducing the multiplier $I$, one would like to

\begin{enumerate}

\item compare the $H^{s}$ norms of the solution with the mollified energy $E(Iu(T))$

\item estimate the slow variation of $E(Iu(T))$ by using Strichartz estimates and a decay looking like

\begin{equation}
\begin{array}{ll}
\int_{|x| \leq R +T} |Iu(T,x)|^{p+1} \, dx & \lesssim  \frac{R}{T+R} E(Iu(0)) + Error \, Terms,
\end{array}
\label{Eqn:ProtMorIu}
\end{equation}
the error terms coming from the fact that the multiplier $I$ does not commute with the nonlinearity.
\end{enumerate}
But before starting the procedure, one must be careful. Indeed, recall that the decay estimate (\ref{Eqn:SubcritMor})
is useful if we work with data compactly supported inside $B(0,R)$. The introduction of the
multiplier $I$ in the Fourier domain kills the localization of the data and consequently, the localization of
the solution inside the cone. But although we cannot perform an analysis inside the cone, we manage to perform an analysis
in a neighborhood of it and outside it in order to control all the error terms: see
Proposition \ref{Prop:SpatSmooth}, Proposition \ref{Prop:DecayPot2} and their proofs. (\ref{Eqn:ProtMorIu}) is established
in Proposition \ref{Prop:DecayPot2}. This is enough to establish (\ref{Eqn:MainIneqLow}): see the beginning of the proof of
Theorem \ref{Thm:Main}. The proof of (\ref{Eqn:ProtMorIu}) relies upon the variation of a mollified energy $E(I_{N_{0}}u)$. One
cannot use this mollified energy in order to find an upper bound of $ \| (P_{>1} u(T), \partial_{t} u(T) \|^{2}_{H^{s} \times H^{s-1}} $
of the solution since

\begin{itemize}

\item
the error appearing in the process of proving (\ref{Eqn:ProtMorIu}) is more difficult to control than that appearing in the process
of estimating the variation of $E(I_{N_{0}} u)$.
\item both errors involve the same parameter $N_{0}$

\end{itemize}
The idea is to introduce a new parameter $N_{1}$, a new mollified energy $E(I_{N_{1}}u)$ and to use
this decay estimate in order to control the variation of $E(I_{N_{1}} u)$: see Section \ref{sec:ProofThm}.

\section{Proof of Theorem \ref{Thm:Main}}
\label{sec:ProofThm}

In this section we prove Theorem \ref{Thm:Main}, assuming that some propositions
are true. We now state these propositions.

This first proposition, proved in Section \ref{Sec:LocalBd}, shows that if we have an a priori bound
of the mollified energy on an interval $J$, then we can control $Z(J,u)$ assuming that $J$ is small
in some sense (see (\ref{Eqn:Cond11}) and (\ref{Eqn:Cond12})):

\begin{prop}{(Local boundedness). }
Let $u$ be a solution of (\ref{Eqn:NlWdat}) on $[0,T]$. Let $J :=[a,b] \subset[0,T]$.
Assume that

\begin{equation}
\begin{array}{ll}
\sup_{t \in J} E(Iu(t)) & \lesssim N^{2{(1-s)}} \cdot
\end{array}
\label{Eqn:InducNrj}
\end{equation}
There exists $\epsilon>0$ small enough such that if

\begin{equation}
\begin{array}{ll}
\| I u \|^{(1-\theta) (p-1)}_{L_{t}^{\infty} L_{x}^{p+1} ( K_{R'} (J) )} |J|^{\frac{(1-\theta)(p-1) +}{p+1}}   N^{(1-s) \theta (p-1)+}
& \leq \epsilon
\end{array}
\label{Eqn:Cond11}
\end{equation}
and

\begin{equation}
\begin{array}{ll}
|J|^{+} \leq \epsilon N^{ \left( (p-1) s - \left(  p-1 - \frac{5-p}{2} \right) \right)+},
\end{array}
\label{Eqn:Cond12}
\end{equation}
then

\begin{equation}
\begin{array}{ll}
Z(J,u) & \lesssim N^{1-s}.
\end{array}
\nonumber
\end{equation}
\label{Prop:Boundedness}
\end{prop}
The second proposition, proved in Section \ref{Sec:PartialDecay} shows that we have a partial decay estimate of the potential term of the
mollified energy. The decrease is partial since only the first term of the right-hand side of
(\ref{Eqn:DecayPot}) shows that there is decay:

\begin{prop}{(Partial decay estimate of potential term of mollified energy).}
Let $(a,b,R') \in \mathbb{R}^{+} \times \mathbb{R}^{+}$. Let $u$ be a solution of (\ref{Eqn:NlWdat}) on
$[a,b]$. Then

\begin{equation}
\begin{array}{ll}
\int_{|x| \leq b + R' } |Iu(b,x)|^{p+1} \, dx  &  \lesssim
\frac{a+R'}{b+R'} E(Iu(a))
+ \frac{1}{\sqrt{2} (b+R')} \int_{ \partial K_{R'}([a,b])}
 \frac{ | \nabla I u \cdot x + (t+R') \partial_{t} I u + I u |^{2} }{t+R'} \, d \sigma \\
& +  \frac{1}{b+R'} \int_{K_{R'}([a,b])}
\Re \left(
\begin{array}{l}
\overline{ (t+R') \partial_{t} Iu  + x \cdot \nabla Iu  + Iu } \\
\left( |Iu|^{p-1} Iu - I(|u|^{p-1} u)  \right)
\end{array}
\right)
\, dz.
\end{array}
\label{Eqn:DecayPot}
\end{equation}
\label{Prop:DecayPot}
\end{prop}
The third proposition, proved in Section \ref{Sec:EstInt}, in an estimate of an integral

\begin{prop}{(Estimate of integral).}
Let $J$ be an interval and let $w$ be a function. Then
\begin{equation}
\begin{array}{ll}
\| I (|w|^{p-1} w) - |Iw|^{p-1} Iw \|_{L_{t}^{1} L_{x}^{2} (J)} & \lesssim \frac{Z^{p}(J,w) |J|^{+}}{N^{\frac{5-p}{2}-}}.
\end{array}
\label{Eqn:EstIw3}
\end{equation}
\label{Prop:EstIw3}
\end{prop}
The fourth proposition, proved in Section \ref{Sec:SpatSmooth}, shows that if a function
is localized, then its smoothness is also more or less localized:

\begin{prop}{(Spatial concentration of smoothness).}
Let $(R_{0},L,R_{0}') \in (\mathbb{R}^{+})^{3}$ such that $L \geq N$ and
$R_{0}^{'} - R_{0}  \geq  \frac{L}{N}$. Let $q \geq 1$. Let $f$ be a smooth function supported on the ball $B(O,R_{0})$. Then

\begin{equation}
\begin{array}{ll}
\| I f \|_{L^{q}(|x| \geq R_{0}^{'})} & \lesssim_{\infty -} \frac{1}{N^{\infty-}} \| I f \|_{L^{q}},
\end{array}
\label{Eqn:SpatSmooth1}
\end{equation}
and

\begin{equation}
\begin{array}{ll}
\| \nabla I f \|_{L^{2} (|x| \geq R_{0}^{'})} & \lesssim_{\infty-} \frac{1}{N^{\infty-}}
\left( \| I f \|_{L^{2}} + \| \nabla I f \|_{L^{2}} \right).
\end{array}
\label{Eqn:SpatSmooth2}
\end{equation}
In particular, if $R_{0}':=R_{0}+1$, then

\begin{equation}
\begin{array}{ll}
\| I f \|_{L^{q}} & \sim  \| I f \|_{L^{q} (|x| \leq R_{0}^{'})}.
\end{array}
\label{Eqn:SpatSmooth3}
\end{equation}

\label{Prop:SpatSmooth}
\end{prop}
The last proposition, proved in Section \ref{Sec:FinalDecay}, shows that, for a large class of potential
terms of mollified energies defined by (\ref{Eqn:ChoiceN0}), the decay is total. The proof uses
the results from Proposition \ref{Prop:DecayPot}, Proposition \ref{Prop:EstIw3}, Proposition \ref{Prop:SpatSmooth} and finite
speed of propagation.

\begin{prop}{(Final decay estimate for a class of potential terms of mollified energies).}
Let $u$ be a solution of (\ref{Eqn:NlWdat}) on $[0,T]$. Assume that

\begin{equation}
\begin{array}{l}
N_{0}^{\frac{5-p}{2} -\frac{(1-s)(p + 1)}{1 -\theta}  } \gg  T^{1+}.
\end{array}
\label{Eqn:ChoiceN0}
\end{equation}
Let $t \in [0,T]$. Then we have

\begin{equation}
\begin{array}{ll}
\int_{|x| \leq R' + t} |I_{N_{0}} u(t,x)|^{p+1} \, dx & \lesssim \frac{R'}{R'+t} N_{0}^{2(1-s)}.
\end{array}
\label{Eqn:DecayPot2}
\end{equation}
\label{Prop:DecayPot2}
\end{prop}
We are now in position to prove Theorem \ref{Thm:Main}.\\
We fist estimate $\| P_{<1} u(T) \|^{2}_{H^{s}}$. We have, by finite speed of propagation
and (\ref{Eqn:SpatSmooth3})

\begin{equation}
\begin{array}{ll}
\| P_{< 1} u(T) \|^{2}_{H^{s}} & \lesssim \| I_{N_{0}} u(T) \|_{L^{2}}^{2} \\
& \sim \| I_{N_{0}} u(T) \|^{2}_{L^{2}(B(O,R'+T))} \\
& \lesssim T^{\frac{3p-5}{p+1}} N_{0}^{\frac{4(1-s)}{p+1}} \cdot
\end{array}
\end{equation}
(\ref{Eqn:MainIneqLow}) follows from optimizing the last inequality in $N_{0}$, in view of the constraint
(\ref{Eqn:ChoiceN0}). \\
Next we estimate $\| ( P_{>1} u(T), \partial_{t} u(T ) ) \|^{2}_{H^{s} \times H^{s-1}}$. We define

\begin{equation}
\begin{array}{ll}
F_{T} & = \left\{ t \in [0,T], \, \sup_{t \in [0,T]} E(I_{N_{1}}u(t)) \leq C N_{1}^{2(1-s)}  \right\}
\end{array}
\nonumber
\end{equation}
for $N_{1}$ such that

\begin{equation}
\begin{array}{ll}
N_{1} & = C N_{0}^{ \frac{2(1-s) } { \frac{2 \theta +  p-1 }{1-\theta} s
- \left( \frac{2 \theta + p-1}{1- \theta} - \frac{5-p}{2}  \right) } +}  \langle T \rangle^{+}
\end{array}
\label{Eqn:ChoiceN1}
\end{equation}
and for $C:=C( \| (u_{0},u_{1}) \|_{H^{s} \times H^{s-1}},R)$ fixed and large enough such that all the
estimates below are true. We claim that $F_{T}=[0,T]$. Indeed

\begin{itemize}

\item $0 \in F_{T}$

\item $F_{T}$ is closed by continuity

\item $F_{T}$ is open. Indeed let $\tilde{T} \in [0,T]$. Then there exists $\delta >0$ such that for all
$\tilde{T} \in (T^{'} - \delta, T^{'} + \delta) \cap [0,T]$ such that

\begin{equation}
\begin{array}{ll}
\sup_{t \in [0,\tilde{T}]} E(I_{N_{1}}u(t)) & \leq 2 C N_{1}^{2(1-s)}
\end{array}
\nonumber
\end{equation}
We aim at proving that in fact

\begin{equation}
\begin{array}{ll}
\sup_{t \in [0, \tilde{T}]} E(I_{N_{1}} u(t)) & \leq C N_{1}^{2(1-s)}
\end{array}
\label{Eqn:EstToproveNrj}
\end{equation}
We divide $[0,\tilde{T}]$ into subintervals $(J_{j}=[j-1,j])_{1 \leq j \leq J}$ and we partition
each $J_{j}$ into subintervals $J_{j,k}$ such that
$|J_{j,k}| = \epsilon^{'}  \left( \frac{R' + j}{R'} \right)^{1-} N_{1}^{\frac{(s-1) \theta (p+1) }{1 - \theta}-} N_{0}^{2(s-1)-}$,
with $\epsilon^{'}$  a fixed constant such that $\epsilon^{'} \ll \epsilon$ a fixed constant ($\epsilon$ is
the constant defined in Proposition \ref{Prop:Boundedness}), except maybe the last one. We see from (\ref{Eqn:SpatSmooth3}) and (\ref{Eqn:DecayPot2}) that

\begin{equation}
\begin{array}{ll}
\| I_{N_{1}} u \|^{(1- \theta) (p-1)}_{L_{t}^{\infty} L_{x}^{p+1} ( K_{R'}(J_{j,k}))} |J_{j,k}|^{\frac{(1-\theta)(p-1)}{p+1}+} N_{1}^{(1-s) \theta (p-1)+} \\
\lesssim  \| I_{N_{1}} u \|^{(1 - \theta) (p-1)}_{L_{t}^{\infty} L_{x}^{p+1} (J_{j,k}) } |J_{j,k}|^{\frac{(1-\theta)(p-1)}{p+1}+} N_{1}^{(1-s) \theta (p-1)+} \\
\lesssim \| I_{N_{0}} u \|^{(1- \theta) (p-1)}_{L_{t}^{\infty} L_{x}^{p+1} (J_{j,k}) } |J_{j,k}|^{\frac{(1-\theta)(p-1)}{p+1}+} N_{1}^{(1-s) \theta (p-1)+} \\
\lesssim \| I_{N_{0}} u \|^{(1 - \theta) (p-1)}_{L_{t}^{\infty} L_{x}^{p+1} (K_{R'} ( J_{j,k}) )} |J_{j,k}|^{\frac{(1-\theta)(p-1)}{p+1}+} N_{1}^{(1-s) \theta (p-1)+} \\
\leq \epsilon,
\end{array}
\end{equation}
 since $N_{1} \ll N_{0}$. By Proposition \ref{Prop:EstIw3} we see that

\begin{equation}
\begin{array}{ll}
|E(I_{N_{1}}u(\tilde{T})) - E(I_{N_{1}}u_{0})| & = \left| \int_{0}^{\tilde{T}} \int \Re \left( \overline{ \partial_{t} I_{N_{1}}u }
\left( |I_{N_{1}}u|^{p-1} I_{N_{1}} u - I_{N_{1}}(|u|^{p-1} u)  \right) \right) \, dx \, dt  \right| \\
& \lesssim \sum_{j,k} \| \partial_{t} I_{N_{1}}u \|_{L_{t}^{\infty} L_{x}^{2} (J_{j,k})} \|  I_{N_{1}}(|u|^{p-1} u)
- |I_{N_{1}}u|^{p-1} I_{N_{1}} u \|_{L_{t}^{1} L_{x}^{2} (J_{j,k})} \\
& \lesssim  \frac{N_{1}^{(p+1)(1-s)} N_{1}^{\frac{(1-s) \theta (p+1)}{1 - \theta}+  } N_{0}^{2(1-s)+}  T^{+} } {N_{1}^{\frac{5-p}{2}-}}  \\
& \ll  N_{1}^{2(1-s)},
\end{array}
\nonumber
\end{equation}
by our choice of $N_{1}$.
\end{itemize}
Therefore we see from this inequality and (\ref{Eqn:BoundInitNrj}) that (\ref{Eqn:EstToproveNrj}) holds. \\
Combining (\ref{Eqn:EstToproveNrj}) with (\ref{Eqn:ChoiceN0}) and (\ref{Eqn:ChoiceN1}) and
optimizing in $N_{0}$ we see that (\ref{Eqn:MainIneqHigh}) holds, by (\ref{Eqn:HighHsNrj}).

\section{Proof of Proposition \ref{Prop:DecayPot}}
\label{Sec:PartialDecay}

In this section we prove Proposition \ref{Prop:DecayPot}.\\
Letting $\tilde{u}$ be such that $\tilde{u}(t+R',x):=u(t,x)$, we see that we may assume, without loss of generality, that $R'=0$. \\

\begin{equation}
\begin{array}{l}
\Re \left( \overline{( t \partial_{t} I u + x \cdot \nabla Iu + Iu )} ( \partial_{tt} Iu - \triangle Iu  + I (|u|^{p-1} u))  \right)
=  \\
\Re \left( \overline{(t \partial_{t} I u + x \cdot \nabla Iu + Iu )}  ( \partial_{tt} Iu - \triangle Iu + |Iu|^{p-1} Iu) \right) \\
+ \Re \left( \overline{( t \partial_{t} Iu + x \cdot Iu + Iu)} ( I (|u|^{p-1} u) - |Iu|^{p-1} Iu ) \right) \\
= A_{1} + A_{2}
\end{array}
\nonumber
\end{equation}
We use an argument from Shatah-Struwe \cite{shatstruwe} to estimate $A_{1}$

\begin{equation}
\begin{array}{ll}
A_{1} & =  \partial_{t} P - \nabla \cdot Q + \frac{p-3}{p+1}  |Iu|^{p+1}
\end{array}
\nonumber
\end{equation}
with

\begin{equation}
\begin{array}{ll}
P & := \frac{t}{2} |\partial_{t} I u|^{2} + \frac{t}{2} |\nabla I u|^{2}
+ \Re \left( (x \cdot \nabla Iu) \overline{\partial_{t} Iu} \right) + t \frac{|Iu|^{p+1}}{p+1} + \Re \left( Iu \overline{\partial_{t} Iu} \right)
\end{array}
\nonumber
\end{equation}
and

\begin{equation}
\begin{array}{ll}
Q & := t \Re \left( \overline{\partial_{t} Iu} \nabla Iu \right) + \Re \left( (x \cdot \nabla Iu) \overline{\nabla Iu} \right)
- \frac{|\nabla Iu|^{2}}{2} x + \frac{|\partial_{t} I u|^{2}}{2} x - x \frac{|Iu|^{p+1}}{p+1} +  \Re \left( \overline{Iu} \nabla Iu \right)
\end{array}
\nonumber
\end{equation}
Integrating by part $A_{1} + A_{2} =0$ inside $K_{0}([a,b])$, we have

\begin{equation}
\begin{array}{l}
\frac{p-3}{p+1} \int_{K_{0}([a,b])} |Iu|^{p+1} \, dz + \int_{K_{0}([a,b])} A_{2} \, dz + H(b)  = H(a) + \frac{1}{\sqrt{2}} \int_{ \partial K_{0} ([a,b])} \left( P + Q \cdot \frac{x}{|x|}  \right) \, d \sigma
\end{array}
\label{Eqn:HtBound0}
\end{equation}
with

\begin{equation}
\begin{array}{ll}
H(t) & = \int_{|x| \leq t} P(t,x) \, dx \cdot
\end{array}
\nonumber
\end{equation}
We have

\begin{equation}
\begin{array}{ll}
H(t) & = \int_{|x| \leq t}
\left[
\begin{array}{l}
\frac{t}{2} \left( |\partial_{t} I u|^{2} + \left| \nabla I u + \frac{Iu x}{|x|^{2}} \right|^{2}   \right) \\
- t \nabla \cdot \left( \frac{|Iu|^{2} x}{2 |x|^{2}} \right)  + \Re \left( \overline{\partial_{t} Iu} ( I u + x \cdot \nabla I u ) \right) \\
+ t \frac{ |Iu|^{p+1}}{p+1}
\end{array}
\right] \, dx
\nonumber
\end{array}
\end{equation}
One one hand, by Hardy's inequality and integration by part, we see that

\begin{equation}
\begin{array}{ll}
H(t) +  \int_{|x| = t} \frac{|Iu|^{2}(t,x)}{2} \, d \sigma \lesssim t E(Iu(t))
\end{array}
\label{Eqn:HtBound1}
\end{equation}
On the other hand, since

\begin{equation}
\begin{array}{ll}
|\partial_{t} I u (Iu + x \cdot \nabla Iu) | & \leq \frac{t}{2}
\left( |\partial_{t} I u|^{2} + \left| \nabla I u + \frac{Iu x}{|x|^{2}}  \right|^{2}   \right), \, |x| \leq t
\end{array}
\nonumber
\end{equation}
we see, after integration by part, that

\begin{equation}
\begin{array}{ll}
t \int_{|x| \leq t} \frac{|Iu|^{p+1}}{p+1} \, dx - \int_{|x| = t} \frac{|Iu|^{2}}{2} \, d \sigma  \leq H(t)
\end{array}
\label{Eqn:HtBound2}
\end{equation}
and, on $\{ (t,x), \, t \in [a,b], \, t=|x| \}$ , we have

\begin{equation}
\begin{array}{ll}
P & = \frac{t}{2} |\partial_{t} I u|^{2} + \frac{t}{2} |\nabla I u|^{2} + \frac{t}{p+1} |Iu|^{p+1}
+ \Re \left( \overline{\partial_{t} Iu} (x \cdot \nabla Iu  + Iu) \right)  \\
& = \frac{t}{2} |\partial_{t} Iu|^{2} + \frac{t}{2} \left( |(Iu)_{r}|^{2}  + \frac{|(Iu)_{\bot}|^{2}}{r^{2}} \right)
+ \frac{t}{p+1} |Iu|^{p+1} + \Re \left( \overline{\partial_{t} Iu} (r (Iu)_{r} + Iu)  \right) \\
& = \frac{t}{2} |\partial_{t} Iu + (Iu)_{r}|^{2} + \frac{1}{2r} | (Iu)_{\bot}|^{2} + \frac{t}{p+1} |Iu|^{p+1} + \Re ( \overline{\partial_{t} Iu} Iu),
\end{array}
\nonumber
\end{equation}

\begin{equation}
\begin{array}{ll}
Q \cdot \frac{x}{|x|} & = t \Re \left( \overline{\partial_{t} Iu} (Iu)_{r} \right) + \frac{r}{2} |(Iu)_{r}|^{2} - \frac{|(I u)_{\bot}|^{2}}{2r}
+ \frac{r |\partial_{t} I u|^{2} }{2} - \frac{r |Iu|^{p+1}}{p+1} + \Re \left( \bar{Iu} (Iu)_{r} \right)
\end{array}
\nonumber
\end{equation}
and

\begin{equation}
\begin{array}{ll}
P + Q \cdot \frac{x}{|x|} & = t | \partial_{t} Iu  + (Iu)_{r} |^{2} + \Re \left( \bar{Iu} (\partial_{t} Iu + (Iu)_{r}) \right) \cdot
\end{array}
\nonumber
\end{equation}
Hence

\begin{equation}
\begin{array}{l}
\frac{1}{\sqrt{2}} \int_{\partial K_{0}([a,b])} P + Q \cdot \frac{x}{|x|} \, d \sigma   \\
=
\frac{1}{\sqrt{2} } \int_{ \partial K_{0} ([a,b])} \frac{ | \nabla I u \cdot x + t \partial_{t} Iu + Iu |^{2}}{t} \, d \sigma
- \frac{1}{\sqrt{2}} \int_{ \partial K_{0} ([a,b])}  \frac{ \Re \left( \bar{Iu} (t \partial_{t} Iu  + x \cdot \nabla Iu ) \right) + |Iu|^{2} }{t} \, d \sigma
\\
= \frac{1}{\sqrt{2} } \int_{ \partial K_{0} ([a,b])} \frac{ | \nabla I u \cdot x + t \partial_{t} Iu + Iu |^{2}}{t} \, d \sigma
- \int_{a \leq |y| \leq b} \nabla \cdot \left( \frac{|Iv|^{2} y}{2 |y|} \right) \, dy \\
= \frac{1}{\sqrt{2} } \int_{\partial K_{0} ([a,b])} \frac{ | \nabla I u \cdot x + t \partial_{t} Iu + Iu |^{2}}{t} \, d \sigma
-  \left( \int_{|y|=b} \frac{|Iu|^{2}(b,y)}{2} \, d \sigma - \int_{|y|=a} \frac{|Iu|^{2}(a,y)}{2} \, d \sigma \right),
\end{array}
\label{Eqn:HtBound3}
\end{equation}
where $v(y):=u(|y|,y)$. We conclude from (\ref{Eqn:HtBound0}), (\ref{Eqn:HtBound1}), (\ref{Eqn:HtBound2}) and (\ref{Eqn:HtBound3}) that
(\ref{Eqn:DecayPot}) holds (with $R=0$).


\section{Proof of Proposition \ref{Prop:Boundedness}}
\label{Sec:LocalBd}

In this section we prove Proposition \ref{Prop:Boundedness}. Throughout the proof we constantly use (\ref{Eqn:SpatSmooth3}).
Let $J=[a,b]$. Let

\begin{equation}
\begin{array}{ll}
\bar{Z}(J,u) & : = \sup \{ Z_{1,s}(J,u), Z_{s_{c},s}(J,u), Z_{ \frac{3}{2} - \frac{5}{2p} ,s }(J,u) \}
\end{array}
\nonumber
\end{equation}
First we estimate $Z_{1,s}(J,u)$. We have

\begin{equation}
\begin{array}{ll}
Z_{1,s}(J,u) & \lesssim \left\| (\partial_{t} I u(a), D I u(a))  \right\|_{L^{2}} + \| |u|^{p-1} u \|_{L_{t}^{1} L_{x}^{2}(J)} \\
& \lesssim N^{1-s} + \| |P_{\ll N} u|^{p-1} P_{ \ll N} u \|_{L_{t}^{1} L_{x}^{2}(J)} +
\| |P_{ \ll N} u|^{p-1} P_{\gtrsim N} u \|_{L_{t}^{1} L_{x}^{2} (J)} \\
& + \| | P_{\gtrsim N} u|^{p-1} P_{\ll N} u \|_{L_{t}^{1} L_{x}^{2} (J)} + \| |P_{\gtrsim N} u|^{p-1} P_{\gtrsim N} u \|_{L_{t}^{1} L_{x}^{2}(J)} \\
& \lesssim  N^{1-s} + A_{1} + A_{2} + A_{3} + A_{4} \cdot
\end{array}
\end{equation}
We deduce from our choice of $R'$, Bernstein's inequality, (\ref{Eqn:Cond11}) and (\ref{Eqn:SpatSmooth3}) that

\begin{equation}
\begin{array}{ll}
A_{1} & \lesssim \| P_{\ll N} u \|^{p-1}_{ L_{t}^{2(p-1)} L_{x}^{2(p-1)} (J)} \| P_{\ll N} u \|_{L_{t}^{2} L_{x}^{\infty} (J)} \\
& \lesssim N^{+} |J|^{+} \| P_{\ll N} u \|^{\theta (p-1)}_{L_{t}^{8} L_{x}^{8} (J)}  \| P_{\ll N} u \|^{(1- \theta)(p-1)}_{L_{t}^{p+1} L_{x}^{p+1} (J)}
\| P_{\ll N} u \|_{L_{t}^{2+} L_{x}^{\infty-} (J) } \\
& \lesssim  N^{+} |J|^{\frac{(1- \theta) (p-1)} {p+1} +} \| I u \|^{\theta(p-1)}_{L_{t}^{8} L_{x}^{8}(J)}
\| I u \|^{(1- \theta)(p-1)}_{L_{t}^{\infty} L_{x}^{p+1} (J) } \| I u \|_{L_{t}^{2+} L_{x}^{\infty-} (J) } \\
& \lesssim N^{+} |J|^{\frac{(1- \theta) (p-1)} {p+1} +} \| I u \|^{\theta(p-1)}_{L_{t}^{8} L_{x}^{8}(J)}
\| I u \|^{(1- \theta)(p-1)}_{L_{t}^{\infty} L_{x}^{p+1} ( K_{R'} (J) ) } \| I u \|_{L_{t}^{2+} L_{x}^{\infty-} (J) }  \\
& \lesssim o \left( \frac{ \bar{Z}^{\theta (p-1)+ 1}(J,u) } {N^{(1-s) \theta(p-1)}} \right)
\end{array}
\nonumber
\end{equation}

\begin{equation}
\begin{array}{ll}
A_{2} & \lesssim  \| P_{\ll N} u \|^{p-1}_{L_{t}^{p-1} L_{x}^{\frac{6(p-1)}{p-3} } (J) }
\| P_{\gtrsim N} u \|_{L_{t}^{\infty} L_{x}^{\frac{6}{6-p}} (J)} \\
& \lesssim N^{+} |J|^{+}    \| P_{\ll N} u \|^{p-1}_{L_{t}^{(p-1)+} L_{x}^{\frac{6(p-1)}{p-3} - } (J) }
\frac{ \| D I u \|_{L_{t}^{\infty} L_{x}^{2}(J)} }{N^{ \frac{5-p}{2}-}} \\
& \lesssim |J|^{+} N^{(1-s)+} \frac{\bar{Z}^{p-1}(J,u)}{N^{\frac{5-p}{2}-}}
\end{array}
\nonumber
\end{equation}

\begin{equation}
\begin{array}{ll}
A_{3} & \lesssim \| P_{\gtrsim N} u \|^{p-1}_{L_{t}^{2(p-1)} L_{x}^{2(p-1)}(J)}  \| P_{\ll N} u \|_{L_{t}^{2} L_{x}^{\infty} (J) } \\
& \lesssim N^{+} |J|^{+} \frac{1}{N^{\frac{5-p}{2}-}}
\| D^{1-s_{c}} I u \|^{p-1}_{L_{t}^{2(p-1)} L_{x}^{2(p-1)} (J) }  \| P_{\ll N} u \|_{L_{t}^{2+} L_{x}^{\infty -} (J)} \\
& \lesssim  \frac{|J|^{+}}{N^{\frac{5-p}{2}-}}  \bar{Z}^{p}(J,u) \\
\end{array}
\nonumber
\end{equation}
and

\begin{equation}
\begin{array}{ll}
A_{4} & \lesssim   \| P_{\gtrsim N} u \|^{p}_{L_{t}^{p} L_{x}^{2p} (J)} \\
& \lesssim  \frac{ \| D^{1- \left( \frac{3}{2} -\frac{5}{2p} \right)  } I u \|^{p}_{L_{t}^{p} L_{x}^{2p} (J)} }  {N^{\frac{5-p}{2} -}} \\
& \lesssim \frac{ Z_{ \frac{3}{2} - \frac{5}{2p} ,s }^{p} (J,u) }{N^{\frac{5-p}{2} -}}
\end{array}
\nonumber
\end{equation}
Then we estimate $Z_{m,s}(J,u)$ for $m \in \left\{ s_{c} , \frac{3}{2} - \frac{5}{2p}  \right\} $. We have

\begin{equation}
\begin{array}{ll}
Z_{m,s}(J,u) & \lesssim \left\| (\partial_{t} I u(a), D I u(a))  \right\|_{L^{2}} + \| D^{1-m} I (|u|^{p-1} u) \|_{L_{t}^{\frac{2}{1+m}} L_{x}^{\frac{2}{2-m}} (J)} \\
& \lesssim N^{1-s} + \| D^{1-m} I u \|_{L_{t}^{\frac{2}{m}} L_{x}^{\frac{2}{1-m}} (J) }
\| u \|^{p-1}_{L_{t}^{2(p-1)} L_{x}^{2(p-1)}(J)} \\
& \lesssim N^{1-s} + Z_{m,s}(J,u) \left( \| P_{\ll N} u \|^{p-1}_{L_{t}^{2(p-1)} L_{x}^{2(p-1)}(J)}  +
\| P_{\gtrsim N} u \|^{p-1}_{L_{t}^{2(p-1)} L_{x}^{2(p-1)} (J)} \right) \\
& \lesssim N^{1-s} +  o \left( \frac{ \bar{Z}^{\theta (p-1) +1}(J,u) } {N^{(1-s) \theta(p-1)  }} \right)
+ \frac{|J|^{+}}{N^{\frac{5-p}{2}-}}  \bar{Z}^{p}(J,u)
\end{array}
\nonumber
\end{equation}
where, at the last line, we used similar arguments to estimate $A_{1}$ and $A_{3}$.\\
Now, by combining all the estimates above and by a continuity argument, we see that $ \bar{Z}(J,u)  \lesssim N^{1-s} $. \\
In particular we see that $Z_{1,s}(J,u) \lesssim N^{1-s}$. By (\ref{Eqn:InducNrj}), we also have $Z_{0,s} (J,u) \lesssim N^{1-s}$.
Now, by interpolating between $m=0$ and $m=1$, we see
that $Z_{m,s}(J,u) \lesssim N^{1-s}$ also holds  if $m \in (0,1)$.

\section{ Proof of Proposition \ref{Prop:SpatSmooth}}
\label{Sec:SpatSmooth}

In this section we prove Proposition \ref{Prop:SpatSmooth}. \\
First we prove the following lemma

\begin{lem}
We have

\begin{equation}
\begin{array}{ll}
\| P_{\ll N} (\chi_{R_{0}} g) \|_{L^{q}(|x| \geq R_{0}^{'})} & \lesssim_{\infty -} \frac{1}{L^{\infty -}} \| g \|_{L^{q}(|x| \leq R'_{0} )}
\end{array}
\label{Eqn:Eq1Lemma}
\end{equation}

\begin{equation}
\begin{array}{ll}
\| \nabla P_{\ll N} (\chi_{R_{0}} g) \|_{L^{2} (|x| \geq R_{0}^{'})} & \lesssim_{\infty -}
\frac{1}{L^{\infty -}}  \left( \| \nabla g \|_{L^{2} (|x| \leq R'_{0} )} + \| g \|_{L^{2} (|x| \leq R'_{0})}  \right)
\end{array}
\label{Eqn:Eq2Lemma}
\end{equation}

\begin{equation}
M \gtrsim N: \;
\begin{array}{ll}
\| P_{M} I (\chi_{R_{0}} g) \|_{L^{q} (|x| \geq R_{0}^{'} )} & \lesssim_{\infty -} \frac{N^{1-s}}{M^{1-s}} \left( \frac{N}{LM} \right)^{\infty -}
\| g \|_{L^{q} (|x| \leq R'_{0}) }
\end{array}
\label{Eqn:Eq3Lemma}
\end{equation}
and

\begin{equation}
M \gtrsim N: \;
\begin{array}{ll}
\| \nabla P_{M} I (\chi_{R_{0}} g) \|_{L^{2} (|x| \geq R_{0}^{'})} & \lesssim_{\infty -} \frac{N^{1-s}}{M^{1-s}}   \left( \frac{N}{LM} \right)^{\infty -}
\left( \| \nabla g \|_{L^{2}(|x| \leq  R'_{0})} +  \| g \|_{L^{2}(|x| \leq R'_{0})} \right)
\end{array}
\label{Eqn:Eq4Lemma}
\end{equation}

\label{Lem:Lem1}
\end{lem}

\begin{proof}
We prove (\ref{Eqn:Eq1Lemma}). We have

\begin{equation}
\begin{array}{ll}
\| P_{\ll N} (\chi_{R_{0}} g) \|_{L^{q} (|x| \geq R'_{0})} & \lesssim
\left\| N^{3} \int \check{\phi} (  128 N y ) \chi_{R_{0}}(x-y) g(x-y)     \right\|_{L^{q} (|x| \geq R_{0}^{'})}  \\
& \lesssim_{\infty -} N^{3} \| g \|_{L^{q}(|x| \leq R'_{0})} \int_{|Ny| \geq N \left( R_0' - \left( R_0 + \frac{1}{2} \right) \right)} \frac{1}{| Ny|^{\infty -}} \, dy \\
& \lesssim_{ \infty -} \frac{1}{L^{\infty-}} \| g \|_{L^{q} (|x| \leq R'_{0})},
\end{array}
\label{Eqn:Calc1}
\end{equation}
the third inequality following from Minkowski's inequality, the supports of the functions, and the fast decay
of $\check{\phi}$. \\
The proof of (\ref{Eqn:Eq2Lemma}) is a straightforward modification of the proof of (\ref{Eqn:Eq1Lemma}): it is left to the reader. \\
In order to prove (\ref{Eqn:Eq3Lemma}) we write

\begin{equation}
\begin{array}{ll}
\widehat{P_{M} I f}(\xi) & = \frac{N^{1-s}}{M^{1-s}} \psi \left( \frac{\xi}{M} \right) \hat{f}(\xi), \, M \gg N  \\
\widehat{P_{M} I f}(\xi) & = \tilde{\psi}^{'} \left( \frac{\xi}{N} \right) \hat{f}(\xi), \, M \sim N
\end{array}
\nonumber
\end{equation}
with $\tilde{\psi}^{'}$ being a localized bump function around $|\xi| \sim 1$ (like $\psi$). Next we
follow the same steps, as in (\ref{Eqn:Calc1}), noticing that $R_{0}^{'} - R_{0} \gtrsim  \frac{\frac{LM}{N}}{M}$.
(\ref{Eqn:Eq4Lemma}) follows easily from (\ref{Eqn:Eq3Lemma}).

\end{proof}
Now we prove (\ref{Eqn:SpatSmooth1}) and (\ref{Eqn:SpatSmooth3}). (\ref{Eqn:SpatSmooth3}) is an easy
consequence of (\ref{Eqn:SpatSmooth1}). So it is enough to prove (\ref{Eqn:SpatSmooth1}). We have

\begin{equation}
\begin{array}{ll}
\|  I f \|_{L^{q}(|x| \geq R_{0}^{'})} & \lesssim
\| P_{\ll N} I (\chi_{R_{0}} f) \|_{L^{q}(|x| \geq R_{0}^{'})} + \| P_{\gtrsim N}  I (\chi_{R_{0}} f ) \|_{L^{q} (|x| \geq R_{0}^{'})} \\
& = A_{1} + A_{2} \cdot
\end{array}
\nonumber
\end{equation}
We deal with $A_{1}$. We have

\begin{equation}
\begin{array}{ll}
A_{1} & \lesssim \| P_{\ll N}  (\chi_{R_{0}} P_{\ll N} f) \|_{L^{q} (|x| \geq R_{0}^{'})}
+ \sum_{M \sim N}  \| P_{\ll N } (\chi_{R_{0}} P_{M} f) \|_{L^{q}(|x| \geq R_{0}^{'})} \\
& + \sum_{M \gg N} \| P_{\ll N } (\chi_{R_{0}} \tilde{P}_{M} P_{M} f) \|_{L^{q}(|x| \geq R_{0}^{'})} \\
& = A_{1,1} + A_{1,2} + A_{1,3}
\end{array}
\nonumber
\end{equation}
So by applying Lemma \ref{Lem:Lem1} we see that

\begin{equation}
\begin{array}{ll}
A_{1,1}, \,  A_{1,2} & \lesssim_{\infty -} \frac{1}{L^{\infty-}} \| I f \|_{L^{q}} \cdot
\end{array}
\nonumber
\end{equation}
Next we turn to $A_{1,3}$. The kernel $K_{M}$ of $P_{\ll
 N}(\chi_{R_{0}} \tilde{P}_{M})$ is

\begin{equation}
\begin{array}{ll}
K_{M}(x,y) & := \int \int \phi \left( \frac{128 \xi}{N} \right) \widehat{\chi_{R_{0}}} (\xi -\eta) \tilde{\psi} \left( \frac{\eta}{M} \right)
e^{i \xi \cdot x} e^{-i \eta \cdot y} \, d \xi \, d \eta \cdot
\end{array}
\label{Eqn:DefK}
\end{equation}
We get the pointwise bound

\begin{equation}
\begin{array}{ll}
|K_{M}(x,y)| & \lesssim_{\infty -} \frac{ \langle R_{0} \rangle^{3}}{M^{\infty -}}
\end{array}
\label{Eqn:Cond1}
\end{equation}
On the other hand, by stationary phase in the direction of $\xi$ we have for all $k \in \mathbb{N}$

\begin{equation}
\begin{array}{ll}
|K_{M}(x,y)| & \lesssim_{k, \infty -} \frac{\langle  R_0 \rangle^{3+k} } { M^{\infty -}  |x|^{k}}
\end{array}
\label{Eqn:Cond2}
\end{equation}
and, by stationary phase in the direction of $\eta$ we have

\begin{equation}
\begin{array}{ll}
|K_{M}(x,y)| & \lesssim_{k, \infty-} \frac{ \langle R_0 \rangle^{3+k}}{M^{\infty -}  |y|^{k}}
\end{array}
\label{Eqn:Cond3}
\end{equation}
Hence we see that

\begin{equation}
\begin{array}{ll}
| K_{M} (x,y)| & \lesssim_{\infty -} \frac{\langle R_{0} \rangle^{3}}{M^{\infty-}} \min{ \left( 1, \frac{1}
{ \left( \frac{|x|}{\langle R_0 \rangle} \right)^{3+} } \right)}  \\
|K_{M}(x,y)| & \lesssim_{\infty -} \frac{\langle R_{0} \rangle^{3}}{M^{\infty-}} \min{ \left( 1, \frac{1}
{ \left( \frac{|y|}{\langle R_0 \rangle} \right)^{3+}} \right) }
\end{array}
\label{Eqn:BoundKM}
\end{equation}
and, by Schur's lemma, we have

\begin{equation}
\begin{array}{ll}
A_{1,3} & \lesssim_{\infty-} \sum_{M \gg N} \frac{1}{M ^{\infty -}} \| P_{M} f \|_{L^{q}} \\
& \lesssim_{\infty-} \frac{1}{N^{\infty-}} \| I f \|_{L^{q}} \cdot
\end{array}
\nonumber
\end{equation}
Now we deal with $A_{2}$. We have

\begin{equation}
\begin{array}{ll}
A_{2} & \lesssim \sum_{M_{1} \gtrsim N, M_{2} \gg M_{1}} \| P_{M_{1}} I ( \chi_{R_{0}} \tilde{P}_{M_{2}} P_{M_{2}} f) \|_{L^{q}(|x| \geq R_{0}^{'})} \\
& + \sum_{M_{1} \gtrsim N} \| P_{M_{1}} I (\chi_{R_{0}} P_{\ll N} f ) \|_{L^{q} (|x| \geq R_{0}^{'})} \\
& + \sum_{M_{1} \gtrsim N,  N \lesssim M_{2} \lesssim M_{1} } \| P_{M_{1}} I ( \chi_{R_{0}} P_{M_{2}} f) \|_{L^{q}(|x| \geq R_{0}^{'})} \\
& = A_{2,1} + A_{2,2} + A_{2,3} \cdot
\end{array}
\nonumber
\end{equation}
We are interested in estimating $A_{2,2}$ and $A_{2,3}$. Using
(\ref{Eqn:Eq3Lemma}) we see, after summation, that

\begin{equation}
\begin{array}{ll}
A_{2,2}, A_{2,3} & \lesssim_{\infty -} \frac{1}{L^{\infty -}} \| I f \|_{L^{q}}
\end{array}
\end{equation}
We are interested in estimating $A_{2,1}$. The kernel of the operator $P_{M_{1}} I (\chi_{R_{0}} \tilde{P}_{M_{2}})$ is

\begin{equation}
\begin{array}{ll}
K_{M_{1},M_{2}}(x,y) & = \int \int \psi \left( \frac{\xi}{M_{1}} \right) \frac{N^{1-s}}{|\xi|^{1-s}}
\widehat{\chi_{R_{0}}}(\xi-\eta) \tilde{\psi} \left( \frac{\eta}{M_{2}} \right) e^{i \xi \cdot x} e^{- i \eta \cdot y} \, d \xi
\, d\eta \cdot
\end{array}
\nonumber
\end{equation}
By slightly modifying the steps from (\ref{Eqn:Cond1}) to (\ref{Eqn:Cond3}) we see that

\begin{equation}
\begin{array}{ll}
|K_{M_{1},M_{2}}(x,y)| & \lesssim_{\infty -}  \frac{\langle R_0 \rangle^{3}}{M_{2}^{\infty -}}  \min{\left( 1,
\frac{1}{\left( \frac{|x|}{\langle R_0 \rangle} \right)^{3+}}  \right) }, \\
|K_{M_{1},M_{2}} (x,y)| & \lesssim_{\infty -}
\frac{\langle R_0 \rangle^{3}}{M_{2}^{\infty -}}  \min{\left( 1,
\frac{1}{\left( \frac{|y|}{\langle R_0 \rangle} \right)^{3+}}  \right) }
\end{array}
\nonumber
\end{equation}
and by Schur's lemma

\begin{equation}
\begin{array}{ll}
A_{2,1} & \lesssim_{\infty-} \sum_{M_{1} \gtrsim N, M_{2} \gg M_{1}} \frac{1}{M_{2}^{\infty-}} \| P_{M_{2}} f \|_{L^{q}} \\
& \lesssim_{\infty -} \frac{1}{M_{2}^{\infty -}}  \| I f \|_{L^{q}} \cdot
\end{array}
\nonumber
\end{equation}
We turn to (\ref{Eqn:SpatSmooth2}). We have

\begin{equation}
\begin{array}{ll}
\| \nabla  I f \|_{L^{2} (|x| \geq R_{0}^{'})} &  \lesssim  \| \nabla P_{\ll N} I (\chi_{R_{0}} f) \|_{L^{2} (|x| \geq R_{0}^{'})}
+ \| \nabla P_{\gtrsim N} I (\chi_{R_{0}} f) \|_{L^{2} (|x| \geq R_{0}^{'})} \\
&  = B_{1} + B_{2}
\end{array}
\end{equation}
By decomposition and the boundedness of the Riesz transforms we see that

\begin{equation}
\begin{array}{ll}
\| \nabla  P_{\ll N} I (\chi_{R_{0}} f) \|_{L^{2}(|x| \geq R_{0}^{'})} & \lesssim
\|  \nabla P_{\ll N}  (\chi_{R_{0}} P_{\ll N} f ) \|_{L^{2}( |x| \geq R_{0}^{'})} \\
& +  \sum_{M \sim N} \| \nabla  P_{\ll N}  (\chi_{R_{0}} P_{M} f) \|_{L^{2} (|x| \geq R_{0}^{'})}  \\
& + \sum_{M \gg N} \| D P_{\ll N}  (\chi_{R_{0}} \tilde{P}_{M} P_{M}) \|_{L^{2}} \\
& = B_{1,1} + B_{1,2} + B_{1,3} \cdot
\end{array}
\nonumber
\end{equation}
We are interested in estimating $B_{1,1}$ and $B_{1,2}$. We see from (\ref{Eqn:Eq2Lemma}), Lemma \ref{Lem:Lem1}, that

\begin{equation}
\begin{array}{ll}
B_{1,1} & \lesssim_{\infty -} \frac{1}{L^{\infty-}} \left( \| \nabla I  f  \|_{L^{2}} + \| I f \|_{L^{2}} \right)
\end{array}
\nonumber
\end{equation}
and

\begin{equation}
\begin{array}{ll}
B_{1,2} & \lesssim_{\infty -} \sum_{M \sim N} \frac{1}{L^{\infty -}} \left( \| \nabla P_{M} f  \|_{L^{2}} +  \| P_{M} f \|_{L^{2}} \right)       \\
& \lesssim_{\infty-} \frac{1}{L^{\infty-}} \| \nabla I f \|_{L^{2}} \cdot
\end{array}
\nonumber
\end{equation}
We are interested in estimating $B_{1,3}$. Again, the kernel of the operator $D P_{\ll N} (\chi_{R_{0}} \tilde{P}_{M})$ is given by

\begin{equation}
\begin{array}{ll}
K_{M}(x,y) & := N \int \int \tilde{\phi} \left( \frac{128 \xi}{N} \right) \widehat{\chi_{R_{0}}} (\xi - \eta)
\hat{\psi} \left( \frac{\eta}{M} \right) e^{i \xi \cdot x} e^{- i \eta \cdot y} \, d \xi \, d \eta,
\end{array}
\nonumber
\end{equation}
where $\tilde{\phi}$ is a localized bump function around $B(O,1)$, like $\phi$. By repeating the steps from
(\ref{Eqn:Cond1}) to (\ref{Eqn:Cond3}) we see that (\ref{Eqn:BoundKM}) holds. Therefore, by Schur's lemma

\begin{equation}
\begin{array}{ll}
B_{1,3} & \lesssim_{\infty -} \sum_{M \gg N} \frac{1}{M^{\infty-}} \| P_{M} f \|_{L^{2}} \\
& \lesssim_{\infty -} \frac{1}{N^{\infty-}} \| \nabla I f \|_{L^{2}} \cdot
\end{array}
\nonumber
\end{equation}
Now we deal with $B_{2}$. We have

\begin{equation}
\begin{array}{ll}
\| \nabla P_{\gtrsim N} I (\chi_{R_{0}} f ) \|_{L^{2}(|x| \geq R_{0}^{'} )} & \lesssim  \sum_{M_{1} \gtrsim N, M_{2} \gg M_{1}}
\|D P_{M_{1}} I ( \chi_{R_{0}} \tilde{P}_{M_{2}} P_{M_{2}} f) \|_{L^{2}} \\
& + \sum_{M_{1} \gtrsim N} \| \nabla P_{M_{1}} I (\chi_{R_{0}} P_{ \ll N} f ) \|_{L^{2} (|x| \geq R_{0}^{'})} \\
& + \sum_{M_{1} \gtrsim N,  N \lesssim M_{2} \lesssim M_{1} } \| \nabla P_{M_{1}} I ( \chi_{R_{0}} P_{M_{2}} f) \|_{L^{2}(|x| \geq R_{0}^{'})} \\
& = B_{2,1} + B_{2,2} + B_{2,3} \cdot
\end{array}
\nonumber
\end{equation}
We see from (\ref{Eqn:Eq4Lemma}) that

\begin{equation}
\begin{array}{ll}
B_{2,2}  & \lesssim_{\infty -} \frac{1}{L^{\infty-}} \left( \| \nabla I f \|_{L^{2}} + \| I f \|_{L^{2}}
 \right), \\
B_{2,3} & \lesssim_{\infty -} \frac{1}{L^{\infty -}} \| \nabla I f \|_{L^{2}} \cdot
 \end{array}
\nonumber
\end{equation}
In order to estimate $B_{2,1}$, we follow similar steps to those to estimate $A_{2,1}$. We find

\begin{equation}
\begin{array}{ll}
B_{2,1} & \lesssim_{\infty -} \frac{1}{M_2^{\infty -}} \| \nabla I f \|_{L^{2}} \cdot
\end{array}
\nonumber
\end{equation}

\section{Proof of Proposition \ref{Prop:DecayPot2}}
\label{Sec:FinalDecay}

In this section we prove Proposition \ref{Prop:DecayPot2}.\\
We define the following set

\begin{equation}
\begin{array}{ll}
F_{T} & := \left\{ T^{'} \in [0,T], \,
\begin{array}{l}
\sup_{t \in [0,T^{'}]} E(I_{N_{0}}u(t)) \leq C_{1} N_{0}^{2(1-s)} \\
\int_{|x| \leq t + R'} | I_{N_{0}} u(t,x) |^{p+1} \, dx  \leq C_{2} \frac{R'}{t + R'} N_{0}^{2(1-s)}, t \in [0,T^{'}]
\end{array}
\right\} \cdot
\end{array}
\nonumber
\end{equation}
We claim that $F_{T}=[0,T]$ for some constants $C_{1}:=C_{1}( \| (u_{0},u_{1}) \|_{H^{s} \times H^{s-1}}) >0$,
$C_{2}:= C_{2} (\| (u_{0},u_{1}) \|_{H^{s} \times H^{s-1}}) > 0 $ fixed and large enough such that all the
estimates below are true. Indeed

\begin{itemize}

\item  $F_{T} \neq \emptyset$: indeed $0 \in F_{T}$ by (\ref{Eqn:BoundInitNrj}) and the elementary estimate
$\| I_{N_{0}} u_{0} \|^{p+1}_{L^{p+1}} \lesssim E(I_{N_{0}}u_{0})$,

\item $F_{T}$ is closed by continuity,

\item $F_{T}$ is open. Indeed let $T^{'} \in F_{T}$. Then, by continuity in time, there exists $\delta >0$ such that
for all $\tilde{T} \in (T^{'} - \delta, T^{'} + \delta) \cap [0,T]$  and for all $t \in [0, \tilde{T}]$

\begin{equation}
\begin{array}{ll}
\sup_{t \in [0,\tilde{T}] } E(I_{N_{0}} u(t)) \leq 2 C_{1} N_{0}^{2(1-s)}
\end{array}
\label{Eqn:AprioriEstNrj}
\end{equation}
and

\begin{equation}
\begin{array}{ll}
\int_{|x| \leq t + R'} |I_{N_{0}} u(t,x)|^{p+1} \, dx & \leq 2 C_{2} \frac{R'}{R' + t} N_{0}^{2(1-s)} \cdot
\end{array}
\label{Eqn:AprioriEstPotNrj}
\end{equation}
We aim at proving that in fact

\begin{equation}
\begin{array}{ll}
\sup_{t \in [0,\tilde{T}] } E(I_{N_{0}} u(t)) \leq  C_{1} N_{0}^{2(1-s)}
\end{array}
\label{Eqn:EstNrj}
\end{equation}
and

\begin{equation}
\begin{array}{ll}
\int_{|x| \leq t + R'} |I_{N_{0}} u(t,x)|^{p+1} \, dx & \leq  C_{2} \frac{R'}{R' + t} N_{0}^{2(1-s)}
\end{array}
\label{Eqn:EstPotNrj}
\end{equation}
To this end

\begin{itemize}
\item we divide $[0, \tilde{T}]$  into subintervals $(J_{j}:=[j-1,j])_{ 1 \leq j \leq J }$, except maybe the last one.
\item we  partition each $J_{j}$ into subintervals $(J_{j,k})_{1 \leq  k \leq K}$ such that
$|J_{j,k}| = \tilde{\epsilon}  \left( \frac{j + R'}{R'} \right)^{1-} N_{0}^{\frac{s-1}{1-\theta} \left( (p-1) \theta +2  \right)- }$, with
$\tilde{\epsilon}$ a fixed positive constant such that $\tilde{\epsilon} \ll \epsilon$ ($\epsilon$ is defined in
Proposition \ref{Prop:Boundedness}), except maybe the last one.
\end{itemize}
Notice that there are at most $ \sim_{R'} \langle \tilde{T} \rangle^{+}  N_{0}^{\frac{1-s}{1- \theta}  \left( \theta (p-1) + 2 \right)+ }  $ subintervals $J_{j,k}$. Notice also from
(\ref{Eqn:AprioriEstNrj}), (\ref{Eqn:EstPotNrj}) and Proposition \ref{Prop:Boundedness} that $Z(J_{j,k},u) \lesssim N_{0}^{1-s}$.
Therefore we get after iterating Proposition \ref{Prop:EstIw3} over $j$ and $k$

\begin{equation}
\begin{array}{ll}
\| I_{N_{0}} (|u|^{p-1} u) - |I_{N_{0}}u|^{p-1} I_{N_{0}} u \|_{L_{t}^{1} L_{x}^{2} ([0,\tilde{T}])} & \lesssim
\frac{  \langle \tilde{T} \rangle^{+} N_{0}^{\frac{1-s}{1- \theta}  \left( \theta (p-1) + 2 \right)+ }  N_{0}^{p(1-s)}} {N_{0}^{\frac{5-p}{2}-}}
\end{array}
\label{Eqn:EstDiffIw3}
\end{equation}
and, for all $t \in [0, \tilde{T}]$

\begin{equation}
\begin{array}{ll}
| E(I_{N_{0}}u(t)) - E(I_{N_{0}}u_{0}) | & \lesssim  \| \partial_{t} I_{N_{0}} u \|_{L_{t}^{\infty} L_{x}^{2} ([0,t])}
\| I_{N_{0}} (|u|^{p-1} u) - |I_{N_{0}}u|^{p-1} I_{N_{0}} u \|_{L_{t}^{1} L_{x}^{2} ([0,t])} \\
& \lesssim \frac{  \langle \tilde{T} \rangle^{+} N_{0}^{\frac{1-s}{1- \theta}  \left( \theta (p-1) + 2 \right)+ }  N_{0}^{(p+1)(1-s)} }{N_{0}^{\frac{5-p}{2}-}} \\
& \leq \frac{1}{100} C_{E} N_{0}^{2(1-s)},
\end{array}
\nonumber
\end{equation}
the last inequality following from our choice of $N_{0}$: see (\ref{Eqn:ChoiceN0}). Therefore
(\ref{Eqn:EstNrj}) holds. \\
Next we turn to the proof of (\ref{Eqn:EstPotNrj}). Using (\ref{Eqn:DecayPot}) with $a:=0$ and $b:=\tilde{T}$  we have

\begin{equation}
\begin{array}{ll}
\int_{|x| \leq \tilde{T} + R'} |I_{N_{0}} u(\tilde{T},x)|^{p+1} \, dx & \lesssim \frac{R'}{\tilde{T} + R'} N_{0}^{2(1-s)} +
X_{1} + \sum_{j,k} X_{2,j,k}
\end{array}
\nonumber
\end{equation}
with

\begin{equation}
\begin{array}{ll}
X_{1} & := \frac{1}{\sqrt{2} (\tilde{T}+R')} \int_{ \partial K_{R'}([0,\tilde{T}])}
 \frac{| \nabla I_{N_{0}} u \cdot x + (t+R') \partial_{t} I_{N_{0}} u + I_{N_{0}} u|^{2} }{t+R'} \, d \sigma
\end{array}
\nonumber
\end{equation}
and

\begin{equation}
\begin{array}{ll}
X_{2,j,k} & := \frac{1}{\tilde{T} + R'} \int_{K_{R'}(J_{j,k})}  \left[
\begin{array}{l}
\Re \overline{ \left( (t+R') \partial_{t} I_{N_{0}}u  + x \cdot\nabla I_{N_{0}}u  + I_{N_{0}}u  \right)}  \\
\left( |I_{N_{0}}u|^{p-1} I_{N_{0}} u - I_{N_{0}} (|u|^{p-1}u)  \right)
\end{array}
\right]  dz
\end{array}
\nonumber
\end{equation}
First we estimate $X_{1}$. We write

\begin{equation}
\begin{array}{ll}
X_{1} & \lesssim \int_{\partial K_{R'}([0,\tilde{T}])} \left| \frac{\nabla I_{N_{0}} u \cdot x}{|x|} + \partial_{t} I_{N_{0}} u \right|^{2} \, d \sigma +
\int_{\partial K_{R^{'}} ([0,\tilde{T}])}  \frac{|I_{N_{0}} u|^{2}}{|t + R'|^{2}} \,  d \sigma \\
& \lesssim Flux(I_{N_{0}} u, \partial K_{R'}([0,\tilde{T}])) + Flux^{\frac{2}{p+1}}(I_{N_{0}} u, \partial K_{R'} ([0,\tilde{T}]))
\left( \frac{1}{R'^{\frac{5-p}{p-1}}} - \frac{1}{(\tilde{T}+R')^{\frac{5-p}{p-1}}} \right)^{\frac{1}{2}}
\end{array}
\end{equation}
where we applied H\"older inequality at the last step.

\begin{equation}
\begin{array}{ll}
0 & = \Re \left( \overline{\partial_{t} I_{N_{0}}u} ( \partial_{tt} I_{N_{0}}u - \triangle I_{N_{0}}u - I_{N_{0}} (|u|^{p-1} u)) \right) \\
& =  \Re \left( \overline{\partial_{t} I_{N_{0}}u} ( \partial_{tt} I_{N_{0}}u - \triangle I_{N_{0}}u - |I_{N_{0}}u|^{p-1} I_{N_{0}} u ) \right) +
\Re \left( \overline{\partial_{t} I_{N_{0}}u} ( |I_{N_{0}}u|^{p-1} I_{N_{0}} u - I_{N_{0}} (|u|^{p-1} u)) \right)
\end{array}
\nonumber
\end{equation}
we see that, after integration by part of this identity
on $K^{c}_{R^{'}} ([0,\tilde{T}])$ that

\begin{equation}
\begin{array}{ll}
Flux(I_{N_{0}} u, \partial K_{R'}([0,\tilde{T}])) + E_{R^{'},ext}(I_{N_{0}}u(\tilde{T})) & = E_{R^{'},ext}(I_{N_{0}}u_{0})  \\
&  + \int_{K^{c}_{R^{'}} ([0,\tilde{T}])} \Re \left( \overline{ \partial_{t} I_{N_{0}}u}
\left( I_{N_{0}}(|u|^{p-1}u) - |I_{N_{0}}u|^{p-1} I_{N_{0}} u \right) \right) \, dz \\
& = Z_{1,1} + Z_{1,2} \cdot
\end{array}
\nonumber
\end{equation}
We estimate $Z_{1,1}$ and $Z_{1,2}$. By (\ref{Eqn:SpatSmooth2}) and (\ref{Eqn:EstNrj})
we have

\begin{equation}
\begin{array}{ll}
\| \nabla I_{N_{0}} u_{0} \|_{L^{2}(|x| \geq R^{'})} & \lesssim_{\infty -}
\frac{1}{N_{0}^{\infty-}} \left( \| I_{N_{0}} u_{0} \|_{L^{2}} + \| \nabla I_{N_{0}} u_{0} \|_{L^{2}} \right) \\
& \lesssim_{\infty-} \frac{1}{N_{0}^{\infty -}}
\end{array}
\nonumber
\end{equation}
Similarly, by (\ref{Eqn:SpatSmooth1}), (\ref{Eqn:EstNrj}) and the finite speed of propagation we  have

\begin{equation}
\begin{array}{ll}
\| I_{N_{0}} u_{0} \|_{L^{p+1} (|x| \geq R')} & \lesssim_{\infty -} \frac{1}{N_{0}^{\infty-}}, \\
\| \partial_{t} I_{N_{0}} u(t) \|_{L^{2} (|x| \geq R^{'} + t)} & \lesssim_{\infty-}  \frac{1}{N_{0}^{\infty-}} \cdot
\end{array}
\nonumber
\end{equation}
Therefore

\begin{equation}
\begin{array}{ll}
Z_{1,1} & \lesssim_{\infty-} \frac{1}{N_{0}^{\infty-}}
\end{array}
\nonumber
\end{equation}
and, using also (\ref{Eqn:EstDiffIw3}), we see that

\begin{equation}
\begin{array}{ll}
Z_{1,2} & \lesssim \| \partial_{t} I_{N_{0}} u \|_{L_{t}^{\infty} L_{x}^{2} (K^{c}_{R^{'}} ([0,\tilde{T}]) )}
\| I_{N_{0}} (|u|^{p-1} u) - |I_{N_{0}}u|^{p-1} I_{N_{0}} u \|_{L_{t}^{1} L_{x}^{2}([0,\tilde{T}])} \\
& \lesssim_{\infty-} \frac{1}{N_{0}^{\infty-}}
\end{array}
\end{equation}
We turn to $X_{2,j,k}$. By H\"older inequality and (\ref{Eqn:AprioriEstPotNrj}) we see that, for
$t \in J_{j,k}$
\begin{equation}
\begin{array}{ll}
\| I_{N_{0}} u(t) \|_{L_{x}^{2} (|x| < t + R^{'})}
& \lesssim (t+R')^{\frac{3(p-1)}{2(p+1)}} \| I_{N_{0}} u(t) \|_{L_{x}^{p+1} (|x| < t + R')} \\
& \lesssim (t+R')^{\frac{3(p-1)}{2(p+1)}} N_{0}^{\frac{2(1-s)}{p+1}}
\end{array}
\nonumber
\end{equation}
and, combining this estimate with Proposition \ref{Prop:EstIw3} we see that

\begin{equation}
\begin{array}{ll}
|X_{2,j,k}| & \lesssim
\| ( \partial_{t} I_{N_{0}} u, \nabla I_{N_{0}} u, I_{N_{0}} u ) \|_{ (L_{t}^{\infty} L_{x}^{2} (J_{j,k}))^{3} }
\| I_{N_{0}} ( |u|^{p-1} u) -  |I_{N_{0}}u|^{p-1} I_{N_{0}} u \|_{L_{t}^{1} L_{x}^{2} (J_{j,k})} \\
& \lesssim \frac{N_{0}^{(p+1)(1-s)} } {N_{0}^{\frac{5-p}{2}-}}
\end{array}
\nonumber
\end{equation}
and, after iterating over $j$ and $k$ we get

\begin{equation}
\begin{array}{ll}
\int_{|x| \leq R^{'} + \tilde{T}} |I_{N_{0}} u(\tilde{T},x)|^{p+1} \, dx - C_{pot} \frac{R^{'}}{R^{'} + \tilde{T}} N_{0}^{2(1-s)} &
\lesssim \frac{N_{0}^{(p+1)(1-s)}  \langle \tilde{T} \rangle^{+} N_{0}^{\frac{1-s}{1- \theta}  \left( \theta (p-1) + 2 \right)+ }} {N_{0}^{\frac{5-p}{2}-}} \\
& \ll C_{pot} \frac{R^{'}}{R^{'} + \tilde{T}} N_{0}^{2(1-s)}
\end{array}
\nonumber
\end{equation}
with $C_{pot}$ constant determined by (\ref{Eqn:DecayPot}). So (\ref{Eqn:EstPotNrj}) holds.
\end{itemize}

\section{Proof of Proposition \ref{Prop:EstIw3}}
\label{Sec:EstInt}

In this section we prove Proposition \ref{Prop:EstIw3}.
Let $F(x):=|x|^{p-1} x$. Then

\begin{equation}
\begin{array}{ll}
\| I F(w) -  F(Iw) \|_{L_{t}^{1} L_{x}^{2}(J)} & \lesssim \| IF(w) - F(w) \|_{L_{t}^{1} L_{x}^{2} (J)} +
\| F(w) - F(Iw) \|_{L_{t}^{1} L_{x}^{2} (J)} \\
& = X_{1} + X_{2}
\end{array}
\nonumber
\end{equation}
We estimate $X_{1}$ (see \cite{troynlkgscatt} for a similar argument):

\begin{equation}
\begin{array}{ll}
F (w) & = F(P_{\ll N} w + P_{\gtrsim N} w) \\
& = F(P_{\ll N} w) + \left( \int_{0}^{1} |P_{\ll N} w  + s P_{\gtrsim N} w|^{p-1} \, ds \right) P_{\gtrsim N} w \\
& + \left( \int_{0}^{1} \frac{ P_{\ll N} w + s P_{\gtrsim N} w }{_{\overline{ P_{\ll N} w  + s P_{\gtrsim} w  }} }
|P_{\ll N} w + s P_{\gtrsim N} w|^{p-1} \, ds \right) \overline{P_{\gtrsim N} w} \cdot
\end{array}
\nonumber
\end{equation}
Hence

\begin{equation}
\begin{array}{ll}
X_{1} & \lesssim \| P_{\gtrsim N} F(P_{\ll N} w ) \|_{L_{t}^{1} L_{x}^{2} (J)} +
\| P_{\gtrsim N} w |P_{\ll N} w|^{p-1} \|_{L_{t}^{1} L_{x}^{2} (J)} +
\| |P_{\gtrsim N} w|^{p} \|_{L_{t}^{1} L_{x}^{2} (J)} \\
& = X_{1,1} + X_{1,2} + X_{1,3} \cdot
\end{array}
\nonumber
\end{equation}
We have

\begin{equation}
\begin{array}{ll}
X_{1,1} & \lesssim \frac{1}{N} \| \nabla F(P_{\ll N} w ) \|_{L_{t}^{1} L_{x}^{2} (J)} \\
& \lesssim \frac{1}{N} \| P_{\ll N } w \|^{p-1}_{L_{t}^{\frac{4(p-1)}{7-p}} L_{x}^{\frac{4(p-1)}{p-3}} (J)}
\| \nabla P_{\ll N} w \|_{L_{t}^{\frac{4}{p-3}} L_{x}^{\frac{4}{5-p}} (J)} \\
& \lesssim   \frac{1}{N^{\frac{5-p}{2}-}} |J|^{+}  \| I w  \|^{p-1}_{L_{t}^{\frac{4(p-1)}{7-p}+} L_{x}^{\frac{4(p-1)}{p-3}-} (J)}
\| D^{1- \left( \frac {p-3}{2} \right)} I w \|_{L_{t}^{\frac{4}{p-3}} L_{x}^{\frac{4}{5-p}} (J)} \\
& \lesssim  \frac{|J|^{+} Z^{p}(J,w)}{N^{\frac{5-p}{2}-}} \cdot
\end{array}
\nonumber
\end{equation}
We estimate $X_{1,2}$, $X_{1,3}$ by using similar arguments to estimate $A_{2}$ and $A_{3}$ respectively in the proof of
Proposition \ref{Prop:Boundedness}. We find

\begin{equation}
\begin{array}{ll}
X_{1,2} & \lesssim  \| P_{\ll N} w \|^{p-1}_{L_{t}^{p-1} L_{x}^{\frac{6(p-1)}{p-3}} (J)} \| P_{\gtrsim N} w \|_{L_{t}^{\infty} L_{x}^{\frac{6}{6-p}} (J)} \\
& \lesssim \frac{|J|^{+} Z^{p}(J,w)}{N^{\frac{5-p}{2}-}}
\end{array}
\end{equation}
and

\begin{equation}
\begin{array}{ll}
X_{1,3} & \lesssim  \| P_{\gtrsim N} w \|^{p}_{L_{t}^{p} L_{x}^{2p} (J)} \\
& \lesssim \frac{|J|^{+} Z^{p}_{\frac{3}{2} - \frac{5}{2p}-,s}(J,w)}{N^{\frac{5-p}{2}-}}  \\
& \lesssim \frac{|J|^{+} Z^{p}(J,w)}{N^{\frac{5-p}{2}-}} \cdot
\end{array}
\nonumber
\end{equation}
We turn to $X_{2}$. We write

\begin{equation}
\begin{array}{ll}
F(w) - F(Iw) &= (w- Iw) O \left( |Iw|^{p-1} + |w|^{p-1} \right) \\
& = (w - Iw)  |P_{\ll N} w|^{p-1} + (w-Iw) | P_{\gtrsim N} I w |^{p-1} + (w-Iw) | P_{\gtrsim N}  w|^{p-1} \\
& = Y_{1} + Y_{2} + Y_{3} \cdot
\end{array}
\nonumber
\end{equation}
We estimate $Y_{1}$ by using similar arguments to estimate $X_{1,2}$ and we estimate $Y_{2}$ and $Y_{3}$
by using similar arguments to estimate $X_{1,3}$. We find

\begin{equation}
\begin{array}{ll}
X_{2} & \lesssim \frac{ |J|^{+} Z^{p}(J,w)}{N^{\frac{5-p}{2}-}} \cdot
\end{array}
\end{equation}
So we proved (\ref{Eqn:EstIw3}).

\end{document}